\pdfoutput=1
\documentclass[10pt]{amsart}
\usepackage{amssymb}
\usepackage{amscd}
\usepackage{graphicx}

\numberwithin{equation}{section}

\def\today{\number\day\space\ifcase\month\or   January\or February\or
  March\or April\or May\or June\or   July\or August\or September\or
  October\or November\or December\fi\   \number\year}

\theoremstyle{definition}
\newtheorem{thm}{Theorem}[section]
\newtheorem{lem}[thm]{Lemma}
\newtheorem{prp}[thm]{Proposition}
\newtheorem{dfn}[thm]{Definition}
\newtheorem{cor}[thm]{Corollary}

\newtheorem{rmk}[thm]{Remark}

\newtheorem{exa}[thm]{Example}

\newtheorem{qst}[thm]{Question}

\newcommand{\beq}{\begin{equation}}
\newcommand{\eeq}{\end{equation}}
\newcommand{\beqr}{\begin{eqnarray*}}
\newcommand{\eeqr}{\end{eqnarray*}}
\newcommand{\bal}{\begin{align*}}
\newcommand{\eal}{\end{align*}}
\newcommand{\bei}{\begin{itemize}}
\newcommand{\eei}{\end{itemize}}
\newcommand{\limi}[1]{\lim_{{#1} \to \infty}}

\newcommand{\af}{\alpha}
\newcommand{\bt}{\beta}
\newcommand{\gm}{\gamma}

\newcommand{\ep}{\varepsilon}

\newcommand{\ld}{\lambda}
\newcommand{\sm}{\sigma}
\newcommand{\kp}{\kappa}
\newcommand{\ph}{\varphi}

\newcommand{\rh}{\rho}

\newcommand{\Gm}{\Gamma}

\newcommand{\Ld}{\Lambda}

\newcommand{\Z}{{\mathbb{Z}}}

\newcommand{\C}{{\mathbb{C}}}
\newcommand{\N}{{\mathbb{Z}}_{> 0}}
\newcommand{\Nz}{{\mathbb{Z}}_{\geq 0}}

\newcommand{\id}{{\mathrm{id}}}

\newcommand{\sint}{{\mathrm{int}}}

\newcommand{\dist}{{\mathrm{dist}}}

\newcommand{\Prim}{{\mathrm{Prim}}}

\newcommand{\card}{{\mathrm{card}}}
\newcommand{\Aut}{{\mathrm{Aut}}}

\newcommand{\dirlim}{\varinjlim}

\newcommand{\andeqn}{\,\,\,\,\,\, {\mbox{and}} \,\,\,\,\,\,}

\title[Permanence properties for crossed products]{Permanence properties
for crossed products and fixed point algebras of finite groups}

\author{Cornel Pasnicu}
\author{N.~Christopher Phillips}

\date{16~August 2012}

\address{Department of Mathematics,
      The University of Texas at San Antonio,
      San Antonio TX 78249, USA.}
\email[]{Cornel.Pasnicu@utsa.edu}
\address{Department of Mathematics, University of Oregon,
      Eugene OR 97403-1222, USA.}
\email[]{ncp@darkwing.uoregon.edu}

\thanks{Some of this material is based upon work of the second
   author supported by the US National Science Foundation
   under Grants DMS-0302401, DMS-0701076, and DMS-1101742.}

\begin{document}

\begin{abstract}
Let $\af \colon G \to \Aut (A)$ be an action of a finite group~$G$
on a C*-algebra~$A.$
We present some conditions under which properties of $A$
pass to the crossed product $C^* (G, A, \alpha)$
or the fixed point algebra $A^{\alpha}.$
We mostly consider the ideal property,
the projection property,
topological dimension zero,
and pure infiniteness.
In many of our results,
additional conditions are necessary on the group, the algebra,
or the action.
Sometimes the action must be strongly pointwise outer,
and in a few results it must have the Rokhlin property.
When $G$ is finite abelian,
we prove that crossed products and fixed point algebras by~$G$ preserve
topological dimension zero with no condition on the action.

We give an example to show that
the ideal property and the projection property do not pass to
fixed point algebras (even when the group is ${\mathbb{Z}}_{2}$).
The construction also gives an example
of a C*-algebra~$B$ which does not have the ideal property
but such that $M_{2} (B)$ does have the ideal property;
in fact, $M_{2} (B)$ has the projection property.
\end{abstract}

\maketitle

\indent
In this paper,
we are interested in permanence properties for
crossed products and fixed point algebras by finite groups.
For the most part,
we consider the following loosely related properties:
\begin{itemize}
\item
The ideal property.
\item
The projection property.
\item
Topological dimension zero.
\item
Pure infiniteness for nonsimple C*-algebras.
\end{itemize}
The ideal property for a C*-algebra~$A,$
first defined in the introduction of~\cite{Stv},
requires that every ideal in~$A$
be generated as an ideal by its projections.
The projection property is a strengthening,
introduced in Definition~1 of~\cite{Psn};
it requires that every ideal in~$A$
have an increasing approximate identity consisting of projections.
(See Definition~4.8 of~\cite{CP} for a variation,
suitable for use with nonseparable C*-algebras.)
Topological dimension zero,
defined in~\cite{BP} (see Definition~\ref{D-TDZero311} below),
means that the topology of $\Prim (A)$
has a base consisting of compact open sets.
The main theorem of~\cite{BE}
states, in effect,
that if $X$ is the primitive ideal space of some separable C*-algebra,
then $X$ is the primitive ideal space of an AF~algebra
if and only if the topology of $X$
has a base consisting of compact open sets.
Pure infiniteness for not necessarily simple C*-algebras
is as in Definition 4.1 of~\cite{KR}.
According to Corollary~4.3 of~\cite{PR},
if $A$ is separable and purely infinite,
then the following are equivalent:
\begin{itemize}
\item
${\mathcal{O}}_2 \otimes A$ has real rank zero.
\item
${\mathcal{O}}_2 \otimes A$ has the ideal property.
\item
$A$ has the ideal property.
\item
$A$ has topological dimension zero.
\end{itemize}

The best plausible permanence results
are as follows:
crossed products by arbitrary discrete groups preserve
pure infiniteness,
while crossed products by exact actions of discrete groups
preserve the other three properties when,
except for a finite normal subgroup,
the action is strongly pointwise outer.
We do not give theorems in anything like this generality.
For example,
it remains unknown whether crossed products by arbitrary actions
of finite groups preserve the projection property,
the ideal property,
or pure infiniteness.

Results on pure infiniteness of crossed products
also appear in~\cite{RrSr}.
They mostly have a somewhat different flavor than the
results given here.
Most of them assume essential freeness of the action
on the equivalence classes of irreducible representations
of the C*-algebra~$A,$
together with some kind of paradoxical decomposition
for the action of $G$ on~$A,$
but do not assume that $A$ is purely infinite.

We summarize our results.
Section~\ref{Sec:0019} is devoted to Rokhlin actions of finite groups
on unital C*-algebras.
We show that crossed products and fixed point algebras
of such actions
preserve pure infiniteness,
the class of countable direct limits of finite direct sums
of unital Kirchberg algebras
satisfying the Universal Coefficient Theorem,
and the class of WB~algebras.

In Section~\ref{Sec:2-0303},
we consider strongly pointwise outer actions of finite groups.
We show that crossed products by such actions preserve
the ideal property and the projection property.
One would hope that no condition on the action would be needed,
and that the result would also hold for fixed point algebras.
Both together can't be true:
we exhibit a pointwise outer (but not strongly pointwise outer)
action of~$\Z_2$
on a C*-algebra with the projection property
such that the fixed point algebra does not even have the ideal
property.
A closely related construction gives
a C*-algebra~$A$ such that $M_2 (A)$ has the projection property
but $A$ does not even have the ideal property,
thus giving a negative answer to a question in~\cite{CP}.

Section~\ref{Sec:TopDimZero} is devoted to C*-algebras with
topological dimension zero.
We prove that the crossed product and fixed point algebra
of an arbitrary action of a finite abelian group
on such a C*-algebra again has topological dimension zero.
We obtain the same result for strongly pointwise outer actions
of arbitrary finite groups.

In Section~\ref{Sec:0123},
we consider purely infinite C*-algebras
with finite primitive ideal spaces.
We show that
crossed products and fixed point algebras
of arbitrary actions of finite groups
preserve this class.
They also preserve the class of purely infinite C*-algebras
with composition series in which all subquotients
have finite primitive ideal spaces.
Thus, any example of a crossed product of a purely infinite C*-algebra
by an action of a finite group
which is not purely infinite
must be somewhat complicated.

We use the following conventions and notation throughout this paper.
Ideals in C*-algebras are always closed two sided ideals.
If $\af \colon G \to \Aut (A)$
is an action of a group~$G$ on a C*-algebra~$A,$
then we denote the fixed point algebra by~$A^{\af}.$
We will usually use the same symbol~$\af$
for the action
$g \mapsto \af_{g} |_I$
induced by $\af$ on an invariant ideal $I \subset A,$
and similarly with invariant subalgebras and subquotients,
as well as $M_n (A)$ and similar constructions.

When $G$ is discrete,
the standard unitary in $C^* (G, A, \af)$
(or in $M ( C^* (G, A, \af) )$ when $A$ is not unital)
corresponding to $g \in G$ will be denoted by~$u_g.$
We use the same notation for the standard unitaries
in reduced crossed products.
The standard conditional expectation
from $C^*_{\mathrm{r}} (G, A, \af)$ to~$A$
will usually be denoted by~$E.$

If $A$ is a C*-algebra,
then $A_{+}$ denotes the set of positive elements of~$A.$

We will make repeated use of the following standard fact.

\begin{prp}\label{P-FP303}
Let $A$ be a C*-algebra, let $G$ be a compact group, and let
$\af \colon G \to \Aut (A)$ be an action of $G$ on~$A.$
Then $A^{\af}$
is isomorphic to a corner of $C^* (G, A, \af).$
\end{prp}

\begin{proof}
See the Proposition in~\cite{Rs}.
\end{proof}

\section{Rokhlin actions of finite groups}\label{Sec:0019}

The Rokhlin property for actions of finite groups
is given in Definition 3.1 of~\cite{Iz};
see also an equivalent form
in Definition~1.1 of~\cite{PhtRp1a},
and see the beginning of Section~2 of~\cite{PhtRp1a}
for some indication of the fairly long history of this property.

In this section we give three classes of unital C*-algebras
which are preserved under formation of
crossed products and fixed point algebras by actions of
finite groups which have the Rokhlin property.
They are the unital but not necessarily simple
purely infinite C*-algebras
(as in Definition~4.1 of~\cite{KR}),
the unital countable direct limits of finite direct sums
of Kirchberg algebras satisfying the Universal Coefficient Theorem,
and the unital WB~algebras
(Definition~4.3 of~\cite{CP}; see Definition~\ref{D:D2_0019} below)
which in addition have the ideal property.

A number of other results of the same general type can be found in
Sections 3 and~4 of~\cite{OP}.

The following terminology will be convenient.
It is adapted from part of
Definition~1.5 of~\cite{OP}
(without the requirement of separability)
and from Definitions 2.1 and~2.2 of~\cite{CP}.
If $X$ is a metric space with metric~$\rh,$
then for $x \in X$ and $B \subset X$
we set $\dist (x, B) = \inf_{y \in B} \rh (x, y).$

\begin{dfn}\label{D-LocApp-303}
Let ${\mathcal{C}}$ be a class of C*-algebras.
A {\emph{strong local ${\mathcal{C}}$-algebra}}
is a C*-algebra $A$
such that for every finite set $S \subset A$ and every $\ep > 0,$
there is a C*-algebra $B \in {\mathcal{C}}$
and a homomorphism $\ph \colon B \to A$
(not necessarily injective)
such that $\dist (a, \, \ph (B)) < \ep$ for all $a \in S.$
We also say that $A$
{\emph{can be locally approximated by~${\mathcal{C}}$}}.
\end{dfn}

We will use the following restatement of Theorem~3.2 of~\cite{OP}
in terms of Definition~\ref{D-LocApp-303}.

\begin{thm}[Theorem~3.2 of~\cite{OP}]\label{T-RPLcAp303}
Let $A$ be a unital C*-algebra,
let $G$ be a finite group,
and let $\af \colon G \to \Aut (A)$
be an action with the Rokhlin property.
Then $C^* (G, A, \alpha)$ can be locally approximated by the class of
matrix algebras over corners of~$A.$
\end{thm}

\begin{prp}\label{P:P1_0019}
Let $A$ be a purely infinite unital C*-algebra,
let $G$ be a finite group,
and let $\af \colon G \to \Aut (A)$
be an action with the Rokhlin property.
Then $C^* (G, A, \alpha)$ and
$A^{\alpha}$ are also purely infinite unital C*-algebras.
\end{prp}

% \begin{rmk}\label{R-PrvPI-304}
We do not know of any example of any action at all
of a finite group~$G$
on a purely infinite C*-algebra~$A$ such that the crossed product
is not purely infinite.
If $A$ is simple,
then $C^* (G, A, \alpha)$ is always purely infinite,
regardless of~$\af.$
See Corollary~4.4 of~\cite{JO}.
% \end{rmk}
%
In Section~\ref{Sec:0123},
we will give other conditions under which
the crossed product of a purely infinite C*-algebra
by a finite
group is purely infinite.

\begin{proof}[Proof of Proposition~\ref{P:P1_0019}]
By Proposition~4.17 of~\cite{KR},
hereditary subalgebras of purely infinite C*-algebras
are again purely infinite.
It follows from Theorem~4.23 of~\cite{KR}
that if $B$ is purely infinite and $n \in \N,$
then $M_n (B)$ is purely infinite.
Therefore Theorem~\ref{T-RPLcAp303}
implies that $C^* (G, A, \alpha)$
can be locally approximated by purely infinite C*-algebras
in the sense of Definition~\ref{D-LocApp-303}.
The proof of Proposition 4.18 of~\cite{KR}
(which is stated for direct limits)
shows that any C*-algebra
which can be locally approximated by purely infinite C*-algebras
is itself purely infinite.
So $C^* (G, A, \alpha)$ is purely infinite.

The statement about $A^{\alpha}$ now follows from
Proposition~4.17 of~\cite{KR} and Proposition~\ref{P-FP303}.
\end{proof}

\begin{thm}\label{T:T1_0019}
Let ${\mathcal{C}}$ be the class of unital (separable nuclear)
C*-algebras that are direct limits of sequences of finite direct
sums of Kirchberg C*-algebras
satisfying the Universal Coefficient Theorem.
Let $A \in {\mathcal{C}},$
let $G$ be a finite group,
and let $\af \colon G \to \Aut (A)$
be an action with the Rokhlin property.
Then $C^* (G, A, \alpha)$ and
$A^{\alpha}$ are both in~${\mathcal{C}}.$
\end{thm}

\begin{proof}
Let ${\mathcal{C}}_0$ be the class of
C*-algebras that are finite direct
sums of (not necessarily unital) Kirchberg algebras
satisfying the Universal Coefficient Theorem.
Clearly direct sums, corners, and quotients
of unital algebras in ${\mathcal{C}}_0$
are again unital and in ${\mathcal{C}}_0.$
Theorem~\ref{T-RPLcAp303} therefore
implies that $C^* (G, A, \alpha)$
can be locally approximated by unital algebras in ${\mathcal{C}}_0$
in the sense of Definition~\ref{D-LocApp-303},
with the additional restriction that all the homomorphisms
appearing in Definition~\ref{D-LocApp-303} are injective.
Since $C^* (G, A, \alpha)$ is separable,
it is now easy to see that,
in the sense of Definition~2.1 of~\cite{DP},
the crossed product $C^* (G, A, \alpha)$
has an exhaustive sequence of unital subalgebras in ${\mathcal{C}}_0.$

If $C^* (G, A, \alpha)$ were stable,
we could apply Corollary~5.11 of~\cite{DP}
to conclude that $C^* (G, A, \alpha)$
is a direct limit of algebras in ${\mathcal{C}}_0.$
However, according to Remark~5.13 of~\cite{DP},
the stability assumption in Corollary~5.11 of~\cite{DP}
is not necessary.
We therefore conclude that $C^* (G, A, \alpha)$
is a direct limit of algebras in ${\mathcal{C}}_0.$
Since $C^* (G, A, \alpha)$ is unital,
it follows that $C^* (G, A, \alpha)$
is a direct limit of unital algebras in ${\mathcal{C}}_0.$
Thus $C^* (G, A, \alpha) \in {\mathcal{C}}.$

Since corners of unital algebras in ${\mathcal{C}}_0$
are again unital and in ${\mathcal{C}}_0,$
it is easy to deduce that every corner of $C^* (G, A, \alpha)$
is in ${\mathcal{C}}.$
Proposition~\ref{P-FP303} therefore implies that
$A^{\af} \in {\mathcal{C}}.$
\end{proof}

Theorem~\ref{T:T1_0019} implies a classification result
for crossed products and fixed point algebras
of Rokhlin actions of finite groups on algebras in~${\mathcal{C}}.$
Specifically, the invariant ${\mathrm{Inv}}_{\mathrm{u}}$
described before Lemma~4.3 of~\cite{DP}
classifies such algebras,
because,
by Theorem~5.8 of~\cite{DP}
and the corresponding part of Remark~5.13 of~\cite{DP},
this invariant classifies algebras in~${\mathcal{C}}.$

We now turn to WB~algebras with the ideal property.
We recall the definition:

\begin{dfn}[Definition~4.3 of~\cite{CP}]\label{D:D2_0019}
A C*-algebra $A$ is a {\emph{WB~algebra}} if for
any ideal $I \subset A$ that is generated by its projections,
the extension
\[
0 \longrightarrow I
  \longrightarrow A
  \longrightarrow A/I
  \longrightarrow 0
\]
is quasidiagonal,
that is, there is an approximate identity for~$I$ consisting
of projections $(p_{\ld})_{\ld \in \Ld}$
(not necessarily countable or increasing)
such that
$\lim \| p_{\ld} a - a p_{\ld} \| = 0$ for all $a \in A.$
\end{dfn}

The WLB algebras of~\cite{CP} will be needed.

\begin{dfn}[Definition~4.3 of~\cite{CP}]\label{D-WLB311}
A C*-algebra $A$ is said to be a {\emph{WLB algebra}}
if $A$ has an approximate identity of projections
(not necessarily increasing)
and for every finite subset $F \subset A$
and every $\ep > 0,$
there is a WB~algebra~$B$ and a homomorphism
$\ph \colon B \to A$ such that $\dist( x, \ph (B)) < \ep$
for every $x \in F$
and
``the projections in $F$ can be $\ep$-lifted to projections in~$B$'',
that is, for every projection $p \in F,$
there is a projection $q \in B$
such that $\| p - \ph (q) \| < \ep.$
\end{dfn}

\begin{lem}\label{L-UnitalWB311}
Let $A$ be a unital C*-algebra.
Then $A$ is a WB~algebra if and only if $A$ is a WLB algebra.
\end{lem}

\begin{proof}
It is obvious that unital WB~algebras are WLB algebras,
and the reverse direction
(for unital algebras) follows from Corollary~4.5 of~\cite{CP}.
\end{proof}

Any LB~algebra (Definition~2.2 of~\cite{P2003}) is a WLB algebra,
so that unital LB~algebras are WB~algebras.
The class of LB~algebras contains the GAH algebras
(Definition~2.1 of~\cite{P2001}), and thus the AH~algebras.

\begin{lem}\label{L-CornerWB303}
Let $A$ be a WB~algebra,
and let $p \in A$ be a projection.
Then $p A p$ is a WB~algebra.
\end{lem}

\begin{proof}
Let $J \subset p A p$ be an ideal
which is generated by its projections.
We have to show that
the extension
\[
0 \longrightarrow J
  \longrightarrow p A p
  \longrightarrow p A p / J
  \longrightarrow 0
\]
is quasidiagonal.
Set $I = {\overline{A J A}} \subset A,$ an ideal in~$A$
which is generated by its projections,
and which satisfies $p I p = J.$
The corresponding extension is quasidiagonal by hypothesis,
so Lemma~3.3 of~\cite{CP} gives the desired conclusion.
\end{proof}

\begin{prp}\label{P:P2_0019}
Let $A$ be a unital WB~algebra with the ideal property,
let $G$ be a finite group,
and let $\af \colon G \to \Aut (A)$
be an action with the Rokhlin property.
Then $C^* (G, A, \alpha)$ and
$A^{\alpha}$
are also unital WB~algebras with the ideal property.
\end{prp}

Since the algebras involved are unital,
Lemma~\ref{L-UnitalWB311} implies that we get
an equivalent statement by replacing ``WB~algebra''
everywhere with ``WLB algebra''.

The method of proof does not work for the ideal property alone,
since, by Example~\ref{E:M2} below,
the ideal property is not preserved by passing to corners.
It follows from
Proposition~4.16 of~\cite{Phfgs} and Corollary~\ref{P:IPPP} below
that crossed products by Rokhlin actions of finite groups
on unital C*-algebras
do preserve the ideal property.
It is not clear what happens with fixed point algebras.
Example~\ref{E:TwoStep} shows that
the hypotheses of Corollary~\ref{P:IPPP}
are too weak to ensure that the fixed point algebra
has the ideal property.

\begin{proof}[Proof of Proposition~\ref{P:P2_0019}]
We first prove that
$C^* (G, A, \alpha)$ is a WLB algebra with the ideal property.
We show that $C^* (G, A, \alpha)$ is a WLB algebra
by proving the local approximation property in
Definition~\ref{D-WLB311}.
The approximating algebras $B$ will also have the ideal property,
so it will follow from Corollary~2.5 of~\cite{CP}
that $C^* (G, A, \alpha)$ has the ideal property.

Thus, let $F \subset C^* (G, A, \alpha)$ be finite and let $\ep > 0.$
Choose $\rh > 0$ with $\rh \leq \ep$
and so small that whenever $C$ is a C*-algebra,
$D \subset C$ is a subalgebra, $p \in C$ is a projection,
and $\dist (p, D) < \rh,$
then there is a projection $q \in D$ such that $\| p - q \| < \ep.$
Use Theorem~\ref{T-RPLcAp303}
to find $n \in \N,$ a projection $f \in M_n (A),$
and a homomorphism $\ph \colon f M_n (A) f \to C^* (G, A, \alpha)$
(not necessarily injective)
such that
$\dist \big( a, \, \ph (f M_n (A) f) \big) < \rh$ for all $a \in F.$
Set $B = f M_n (A) f / \ker (\ph),$
and let ${\overline{\ph}} \colon B \to C^* (G, A, \alpha)$
be the map induced by~$\ph.$

It follows from the proof of Lemma~4.14 of~\cite{CP}
that $M_n (A)$ is a WB~algebra,
and obviously $M_n (A)$ has the ideal property.
Therefore $f M_n (A) f$ is a WB~algebra
by Lemma~\ref{L-CornerWB303},
and has the ideal property by Proposition~4.9 of~\cite{CP}.
Since $f M_n (A) f$ has the ideal property,
Remark~4.2 of~\cite{CP} implies that $B$ is a WB~algebra,
and clearly $B$ has the ideal property.
The choice of $\rh$ implies that for every $x \in F$
we have $\dist( x, {\overline{\ph}} (B)) < \ep.$
Since ${\overline{\ph}}$ is injective,
it also follows that for every projection $p \in F,$
there is a projection $q \in B$
such that $\| p - \ph (q) \| < \ep.$
This completes the proof that
$C^* (G, A, \alpha)$ is a WLB algebra with the ideal property.

Lemma~\ref{L-UnitalWB311} now implies that
$C^* (G, A, \alpha)$ is a WB~algebra.

Proposition~\ref{P-FP303} exhibits $A^{\af}$
as a corner of $C^* (G, A, \alpha),$
so it follows from Proposition~4.9 of~\cite{CP}
that $A^{\alpha}$ is also a unital WB~algebra with the ideal property.
\end{proof}

\section{Strongly pointwise outer actions
  and the ideal and projection properties}\label{Sec:2-0303}

\indent
Strongly pointwise outer actions
(see Definition~\ref{D:SPOut} below)
are those for which no group element $g \neq 1$
is inner on any $g$-invariant subquotient.
When the C*-algebra is simple,
this condition is just pointwise outerness.

In this section we show that if $\af$ is a strongly
pointwise outer action of a finite group $G$ on a C*-algebra~$A,$
then $A$ separates the ideals
in $C^* (G, A, \af).$
It follows that such actions are hereditarily saturated
and their crossed products preserve the ideal and projection properties.

It seems plausible that crossed products and
fixed point algebras of arbitrary
actions of finite groups
should preserve the ideal and projection properties.
We have not settled the question for crossed products,
but we show by example that the statement about fixed point algebras
is false,
even when $G = \Z_2.$
Our construction further produces a C*-algebra $B$
such that $M_2 (B)$ has the projection property
but $B$ does not even have the ideal property.
This gives a negative answer to Question~6.8 of~\cite{CP}.

\begin{dfn}[Definition~4.11 of~\cite{Phfgs}]\label{D:SPOut}
An action $\af \colon G \to \Aut (A)$ is said to be
{\emph{strongly pointwise outer}} if,
for every $g \in G \setminus \{ 1 \}$ and any
two $\af_g$-invariant ideals $I \subset J \subset A$
with $I \neq J,$
the automorphism of $J / I$ induced by $\af_g$ is outer.
\end{dfn}

\begin{dfn}[\cite{Sr}]\label{D:S}
Let $\af \colon G \to \Aut (A)$ be an action of a
discrete group $G$ on a C*-algebra~$A.$
We say that $A$
{\emph{separates the ideals in the reduced crossed product}}
$C^*_{\mathrm{r}} (G, A, \alpha)$
(or in $C^* (G, A, \alpha)$ when $G$ is amenable)
if each ideal of $C^*_{\mathrm{r}} (G, A, \alpha)$
has the form $C^*_{\mathrm{r}} ( G, I, \af )$
for some $\af$-invariant ideal $I \subset A.$
\end{dfn}

By Proposition 7.7.9 of~\cite{Pd1},
for every $\af$-invariant ideal $I \subset A,$
the obvious map
$C^*_{\mathrm{r}} ( G, I, \af ) \to C^*_{\mathrm{r}} (G, A, \alpha)$
is injective.
Its image is clearly an ideal.

\begin{thm}\label{P:IdealPrsv}
Let $G$ be a finite group, let $A$ be a C*-algebra,
and let $\af \colon G \to \Aut (A)$
be a strongly pointwise outer action.
Then $A$ separates the ideals in $C^* (G, A, \af).$
\end{thm}

\begin{proof}
Let $J$ be an ideal in $C^* (G, A, \af).$
Set $I = J \cap A.$
We claim that $J$ is an $\af$-invariant ideal in~$A$
such that $C^* ( G, I, \af ) \subset J.$
For $\af$-invariance,
let $a \in I$ and $g \in G.$
Then $\af_g (a) = u_g a u_g^*$ is in both $J$ and~$A.$
If $a \in I$ and $g \in G,$
then $a \in J,$ so $a u_g \in J.$
Thus $C^* ( G, I, \af ) \subset J,$
The claim is proved.

Suppose $C^* ( G, I, \af ) \neq J.$
Let $B = A / I,$
and let $\pi \colon A \to B$ be the quotient map.
Deviating from our usual convention
(to avoid confusion below),
let $\bt \colon G \to \Aut (B)$
be the induced action on the quotient.
Let $\rh \colon C^* (G, A, \af) \to C^* (G, B, \bt)$
be the map induced by~$\pi.$
Then $\ker (\rh) = C^* ( G, I, \af ).$
Let $L = \rh (J).$
Then $L$ is a nonzero ideal of $C^* (G, B, \bt),$
but $L \cap B = \{ 0 \}.$
It follows from Theorem~1.1 of~\cite{RfFG}
that there is $g \in G \setminus \{ 1 \}$ such that
$\bt_g$ is partly inner in the sense of~\cite{RfFG},
that is,
there is some nonzero $\bt_g$-invariant ideal $T \subset B$
such that the restriction of $\bt_g$ to $T$
is inner (in $M (T)$).
This contradicts the assumption that
$\af$ is strongly pointwise outer.
\end{proof}

Examples 4.13 and 4.14 of~\cite{Phfgs}
show that several weaker versions of Definition~\ref{D:SPOut}
do not suffice for the conclusion of Theorem~\ref{P:IdealPrsv}.
In particular,
one must consider subquotients (not just ideals and quotients),
and one must consider subquotients invariant under subgroups,
not just subquotients invariant under the whole group.

\begin{cor}\label{C:Sat}
Let $G$ be a finite group, let $A$ be a C*-algebra,
and let $\af \colon G \to \Aut (A)$
be a strongly pointwise outer action.
Then $\af$ is hereditarily saturated
(Definition~7.2.2 of~\cite{Ph1})
and the strong Connes spectrum ${\widetilde{\Gm}} (\af)$
(Definition~1.2(b) of~\cite{GLP}) is equal to~${\widehat{G}},$
the space of unitary equivalence classes
of irreducible representations of $G.$
\end{cor}

\begin{proof}
This follows from Theorem~\ref{P:IdealPrsv} and~\cite{GLP},
as described in Theorem~5.11 of~\cite{Phfgs}.
\end{proof}

As an another immediate consequence of Theorem~\ref{P:IdealPrsv},
we obtain the following result,
which we originally proved directly from the definition.

\begin{cor}\label{P:IdealFGpRP}
Let $G$ be a finite group, let $A$ be a unital C*-algebra,
and let $\af \colon G \to \Aut (A)$
be an action with the Rokhlin property.
Then $A$ separates the ideals in $C^* (G, A, \af).$
\end{cor}

\begin{proof}
Use Proposition~4.16 of~\cite{Phfgs} and Theorem~\ref{P:IdealPrsv}.
\end{proof}

\begin{cor}\label{P:IPPP}
Crossed products
by strongly pointwise outer actions of finite groups
preserve the ideal property and the projection property.
\end{cor}

\begin{proof}
Let $\af \colon G \to \Aut (A)$ be a strongly pointwise outer action
of a finite group $G$ on a C*-algebra~$A.$

Assume that $A$ has the ideal property.
Let $J \subset C^* (G, A, \alpha)$ be an ideal.
By Theorem~\ref{P:IdealPrsv},
there is a $G$-invariant ideal $I \subset A$
such that $J = C^* (G, I, \alpha).$
By hypothesis, $I$~is generated by its projections,
and it is then easy to see that $C^* (G, I, \alpha)$
is generated by its projections.

The argument for the projection property is the same.
\end{proof}

% \begin{rmk}\label{R-PrsvPP-304}
We do not know of any example of any action at all of a finite group
on a C*-algebra with the ideal property such that the crossed product
does not have the ideal property.
Similarly,
we do not know of any example of any action at all of a finite group
on a C*-algebra with the projection property
such that the crossed product
does not have the projection property.
% \end{rmk}

The following example shows that
the ideal and projection properties do not pass
to fixed point algebras of actions of finite groups.
In fact,
we produce an example of
a pointwise outer (but not strongly pointwise outer)
action of~$\Z_2$
on a C*-algebra with the projection property
such that the fixed point algebra
does not even have the ideal property.

\begin{exa}\label{E:TwoStep}
Let $D$ be the unital Kirchberg algebra satisfying the
Universal Coefficient Theorem and with
$K_0 (D) = 0$ and $K_1 (D) \cong {\mathbb{Z}}_{2}.$
For any C*-algebra~$E,$
let $Q (E) = M (E) / E,$
the corona algebra.
It follows from Theorem~3.3 of~\cite{Z2} that
$Q (K \otimes D)$ is simple,
and from Theorem~1.3 of~\cite{Z1} that
$Q (K \otimes D)$ is purely infinite.
Since $K_* (M (K \otimes D)) = 0$
(by Proposition 12.2.1 of~\cite{Bl}),
the six term exact sequence in K-theory gives
$K_0 (Q (K \otimes D)) \cong {\mathbb{Z}}_{2}$
and $K_1 (Q (K \otimes D)) = 0.$
Let $\pi \colon M (K \otimes D) \to Q (K \otimes D)$
be the quotient map.
Since $1 \in Q (K \otimes D)$ is equal to $\pi (1),$
and $K_0 (M (K \otimes D)) = 0,$
it follows that $[1] = 0$ in $K_0 (Q (K \otimes D)).$

Since $Q (K \otimes D)$ is purely infinite and simple,
there is a projection $e \in Q (K \otimes D)$ whose class
in $K_0 (Q (K \otimes D))$ is the nontrivial element,
and any two such projections are Murray-von Neumann equivalent.
Since $[1 - e] = [1] - [e] = - [e] = [e],$
there is $s \in Q (K \otimes D)$
such that $s^* s = e$ and $s s^* = 1 - e.$
Set $v = s + s^*,$
which is a unitary in $Q (K \otimes D)$ such that
$v^* = v$ and $v e v^* = 1 - e.$
Since $v$ is a selfadjoint unitary,
it is in the identity component
of the unitary group of $Q (K \otimes D),$
and therefore there is a unitary $u \in M (K \otimes D)$
such that $\pi (u) = v.$
(We can't choose $u$ to satisfy $u^2 = 1.$)

Set $C = \C e + \C (1 - e),$
which is a subalgebra of $Q (K \otimes D)$ with $C \cong \C \oplus \C.$
Define a subalgebra $B \subset M (K \otimes D) \oplus M (K \otimes D)$
by
\[
B = \big\{ (b_1, b_2) \in M (K \otimes D) \oplus M (K \otimes D) \colon
  {\mbox{$\pi (b_1), \pi (b_2) \in C$ and $\pi (b_1) = \pi (b_2)$}}
             \big\}.
\]
Let $\kp \colon B \to C$ be the map
$\kp (b_1, b_2) = \pi (b_1).$
Then there is a short exact sequence
\begin{equation}\label{Eq:SES}
0 \longrightarrow K \otimes D \oplus K \otimes D
  \longrightarrow B
  \stackrel{\kp}{\longrightarrow} C
  \longrightarrow 0.
\end{equation}

Define $\gm \in \Aut (C)$ by
\[
\gm \big( \ld_1 e + \ld_2 (1 - e) \big) = \ld_2 e + \ld_1 (1 - e)
\]
for $\ld_1, \ld_2 \in \C.$
Then $v c v^* = \gm (c)$ for all $c \in C.$
Define $\bt \colon B \to B$ by
\[
\bt (b_1, b_2) = (u b_2 u^*, \, u^* b_1 u).
\]
We claim that if $(b_1, b_2) \in B,$
then $\bt (b_1, b_2)$ really is in~$B.$
Indeed, using $v = v^*$ and $\pi (b_1) = \pi (b_2)$
at the second step, we have
\[
\pi (u b_2 u^*)
 = v \pi (b_2) v^*
 = v^* \pi (b_1) v
 = \pi (u b_1 u^*).
\]
Also, $\pi (u b_2 u^*) \in C$ because $v C v^* = C.$
This proves the claim.
One checks immediately that $\bt^2 = \id_B,$
so $\bt$ is invertible
Also,
since for $c \in C$ we have $v c v^* = \gm (c),$
for $(b_1, b_2) \in B$ we get
\[
(\kp \circ \bt) (b_1, b_2)
 = \pi (u b_2 u^*)
 = v \pi (b_2) v^*
 = v \pi (b_1) v^*
 = (\gm \circ \kp) (b_1, b_2).
\]
Let $\ph \in \Aut (K \otimes D \oplus K \otimes D)$
be the restriction of~$\bt.$
Then $\ph,$ $\bt,$ and $\gm$ all define actions of ${\mathbb{Z}}_{2},$
which we denote by the same letters,
and the exact sequence~(\ref{Eq:SES}) is equivariant.

We take crossed products in~(\ref{Eq:SES}).
We have $C^* ({\mathbb{Z}}_{2}, C, \gm) \cong M_2.$
The automorphism $\sm$ of $K \otimes D \oplus K \otimes D$
defined by
\[
\sm (b_1, b_2) = (u^* b_1 u, \, b_2)
\]
satisfies
\[
\big( \sm \circ \bt \circ \sm^{-1} \big)
           (b_1, b_2)
  = (b_2, b_1),
\]
so, defining
$\rh = \sm \circ \bt \circ \sm^{-1},$
we have
\[
C^*
 \big( {\mathbb{Z}}_{2}, \, K \otimes D \oplus K \otimes D, \, \rh \big)
  \cong M_2 \otimes K \otimes D.
\]
Letting $A = C^* ({\mathbb{Z}}_{2}, B, \bt),$
we therefore get a short exact sequence,
equivariant for the dual actions:
\begin{equation}\label{Eq:CPS}
0 \longrightarrow M_2 \otimes K \otimes D
  \longrightarrow A
  \longrightarrow M_2
  \longrightarrow 0.
\end{equation}

We claim that $A$ has the projection property.
{}From the sequence~(\ref{Eq:CPS}),
we see that $A$ has only two nonzero ideals,
namely $M_2 \otimes K \otimes D$ and $A$ itself.
The algebra $A$ is unital,
and $M_2 \otimes K \otimes D$
has an increasing approximate identity consisting of projections
because $D$ is unital.
The claim follows.

We claim that $B$ does not have the ideal property.
To prove this, we set $J = \kp^{-1} (\C e),$
which is an ideal in $B,$
and show that $J$ is not generated by its projections.
It suffices to show that if $p \in J$ is a projection,
then $\kp (p) = 0.$
Write $p = (p_1, p_2)$ with $p_1, p_2 \in M (K \otimes D).$
Then $\pi (p_1)$ is a projection in $\C e,$
and so $\pi (p_1) = e$ or $\pi (p_1) = 0.$
However, $\pi (p_1) = e$ is ruled out by $[p_1] = 0$ in
$K_0 (M (K \otimes D))$ and $[e] \neq 0$ in $K_0 (Q (K \otimes D)).$
So $\kp (p_1, p_2) = \pi (p_1) = 0.$
This proves the claim.

Let $\af = {\widehat{\bt}}$ be the dual action,
and identify ${\widehat{{\mathbb{Z}}_{2}}}$ with ${\mathbb{Z}}_{2}.$
Then $A^{\af} = B.$
Accordingly, we have an action of ${\mathbb{Z}}_{2}$ on a C*-algebra~$A$
with the projection property
(and, in particular, the ideal property)
such that $A^{\af}$ does not have the ideal property
(and, in particular, does not have the projection property).

As preparation for proving that $\af$ is pointwise outer,
we claim that the center of~$A$ is $\C \cdot 1.$
It is clear that $M_2 \otimes K \otimes D$ is an essential ideal in~$A,$
so we can identify $A$
with a subalgebra of $M (M_2 \otimes K \otimes D)$
which contains $M_2 \otimes K \otimes D.$
Since $M_2 \otimes K \otimes D$ is simple,
\[
\big\{ z \in M (M_2 \otimes K \otimes D) \colon
  {\mbox{$z b = b z$ for all $b \in M_2 \otimes K \otimes D$}} \big\}
 = \C \cdot 1.
\]
The claim follows.

Now let $\ph \in \Aut (A)$ be the automorphism of order two
which generates the action~$\af$;
we prove that $\ph$ is outer.
If not,
there is a unitary $u \in A$ such that
$\ph (a) = u a u^*$ for all $a \in A.$
Then $\ph (u) = u,$
so $u^2 a u^{-2} = \ph^2 (a) = a$ for all $a \in A.$
By the previous paragraph,
there is $\ld_0 \in \C$ with $| \ld_0 | = 1$
such that $u^2 = \ld_0 \cdot 1.$
Choose $\ld \in \C$ such that $\ld^{-2} = \ld_0$ and set $v = \ld u.$
Then $v^2 = 1$ and
$\ph (a) = v a v^*$ for all $a \in A.$
So $\af$ is an inner action.
Therefore $A \oplus A \cong C^* (\Z_2, A, \af) \cong M_2 (B).$
Since $B$ has no nontrivial direct sum decomposition,
this contradicts innerness of~$\ph,$
and shows that $\af$ is pointwise outer.
\end{exa}

Question~6.8 of~\cite{CP} asked the following.
Let $A$ be a C*-algebra, let $n \in \N,$
and suppose that $M_{n} (A)$ has the ideal property.
Does it follow that $A$ has the ideal property?
It is also natural to ask the same question
for the projection property in place of the ideal property.
The following example provides negative answers to both questions.

\begin{exa}\label{E:M2}
We use the notation of Example~\ref{E:TwoStep} throughout.
We claim that $M_2 (B)$ does have the projection property.
Thus, $M_2 (B)$ can have the projection property
when $B$ does not even have the ideal property.

To prove the claim, we consider the nonzero ideals in~$B.$
There are six of them:
\[
K \otimes D \oplus \{ 0 \},
\,\,\,\,\,\,
\{ 0 \} \oplus K \otimes D,
\,\,\,\,\,\,
K \otimes D \oplus K \otimes D,
\,\,\,\,\,\,
\kp^{-1} (\C e),
\,\,\,\,\,\,
\kp^{-1} (\C (1 - e)),
\,\,\,\,\,\,
B.
\]
For each ideal $J$ on this list,
we need to show that $M_2 (J)$
has an increasing approximate identity consisting of projections.
For all but $\kp^{-1} (\C e)$ and $\kp^{-1} (\C (1 - e)),$
this is immediate.
Moreover,
with $\bt$ as in Example~\ref{E:TwoStep},
we have $\bt ( \kp^{-1} (\C e) ) = \kp^{-1} (\C (1 - e)).$
Therefore it suffices to consider $\kp^{-1} (\C e).$

By abuse of notation, given a homomorphism $\ph \colon E \to F,$
we use the same letter for the corresponding homomorphism
from $M_2 (E)$ to $M_2 (F).$

Define $w \in M_2 (Q (K \otimes D))$ by
\[
w = \left( \begin{matrix}
  e       &  s        \\
  s^*     &  1 - e
\end{matrix} \right).
\]
One checks that $w$ is a selfadjoint unitary,
and that
\[
w \left( \begin{matrix}
  e     &  0        \\
  0     &  e
\end{matrix} \right)
 w^*
   = \left( \begin{matrix}
  1     &  0        \\
  0     &  0
\end{matrix} \right).
\]
Since $w$ is a selfadjoint unitary,
$w$ is in the
identity component of the unitary group of $M_2 (Q (K \otimes D)).$
Therefore there is a unitary $y \in M_2 (M (K \otimes D))$
such that $\pi (y) = w.$
Define  $p \in M_2 (M (K \otimes D))$ by
\[
p = y^*  \left( \begin{matrix}
  1     &  0        \\
  0     &  0
\end{matrix} \right) y.
\]
Then $p$ is a projection and
\[
\pi (p) = \left( \begin{matrix}
  e     &  0        \\
  0     &  e
\end{matrix} \right).
\]
Moreover,
$1 - p$ is unitarily equivalent to
$1 \oplus 0,$
% $\left( \begin{matrix}
%   1     &  0        \\
%   0     &  0
% \end{matrix} \right),$
from which it follows that
\[
(1 - p) M_2 (K \otimes D) (1 - p) \cong K \otimes D
\]
has an increasing approximate identity $(f_n)_{n \in \N}$
consisting of projections.

Define $q_n \in M_2 (\kp^{-1} (\C e))$
by
\[
q_n = ( p + f_n, \, p + f_n ).
\]
We claim that $(q_n)_{n \in \N}$
is an approximate identity for $M_2 (\kp^{-1} (\C e)).$
We begin by observing that, for any $a \in M_2 (\pi^{-1} (\C e)),$
we have $\pi ((1 - p) a) = 0,$
so
\[
(1 - p) a a^* (1 - p) \in (1 - p) M_2 (K \otimes D) (1 - p).
\]
Thus
\begin{align*}
& \big\| \big[ f_n (1 - p) a - (1 - p) a \big]
  \cdot \big[ f_n (1 - p) a - (1 - p) a \big]^* \big\|
      \\
& \hspace*{6em} {\mbox{}}
 \leq \big\| f_n (1 - p) a a^* (1 - p) - (1 - p) a a^* (1 - p) \big\|
              \cdot \| f_n \|
      \\
& \hspace*{9em} {\mbox{}}
     + \big\| f_n (1 - p) a a^* (1 - p) - (1 - p) a a^* (1 - p) \big\|,
\end{align*}
which converges to zero as $n \to \infty.$
So
\[
\limi{n} \| f_n (1 - p) a - (1 - p) a \| = 0.
\]
Using $f_n (1 - p) = f_n,$
for $a \in M_2 (\pi^{-1} (\C e))$ we therefore get
\[
\limi{n} \big( (p + f_n) a - a \big)
 = \limi{n} \big( f_n (1 - p) a - (1 - p) a \big)
 = 0.
\]
Since
\[
M_2 (\kp^{-1} (\C e))
 = \big\{ (b_1, b_2) \in M_2 (\pi^{-1} (\C e)) \colon
     \pi (b_1) = \pi (b_2) \big\},
\]
we immediately get $\limi{n} (q_n b - b) = 0$
for all $b \in M_2 (\kp^{-1} (\C e)).$
Taking adjoints,
we also get $\limi{n} (b q_n - b) = 0$
for all $b \in M_2 (\kp^{-1} (\C e)).$
So  $(q_n)_{n \in \N}$
is an approximate identity, as desired.
This completes the proof that $M_2 (B)$ has the projection property.
\end{exa}

\section{Topological dimension zero}\label{Sec:TopDimZero}

\indent
In this section,
we prove that arbitrary crossed products
by actions of finite abelian groups
preserve the property topological dimension zero of~\cite{BP}.
(See Definition~\ref{D-TDZero311}.)
See the introduction for some of the significance of this condition.
Our proof depends on duality.
It seems plausible that noncommutative duality could be used
to extend the result to nonabelian groups,
but we have not succeeded in carrying this out.
We do show (with a much easier proof)
that the result holds for nonabelian groups if in addition
the algebra separates the ideals in the crossed product.

In the following,
topological spaces need not be Hausdorff unless otherwise specified.
A compact set is one which has the Heine-Borel property,
regardless of whether or not it is closed or Hausdorff.
(This property is sometimes called quasicompactness.)
We repeat these statements for emphasis
in some of the definitions and lemmas.
We give the following definition to make our terminology clear.

\begin{dfn}[Compare with 3.3.8 of~\cite{Dx}]\label{D-NT2LocCpt311}
Let $X$ be a not necessarily Hausdorff topological space.
We say that $X$ is
{\emph{locally compact}}
if for every $x \in X$ and every open set $U \subset X$
such that $x \in U,$
there exists a compact (but not necessarily closed)
subset $Y \subset X$
such that $x \in \sint (Y) \subset Y \subset U.$
\end{dfn}

Equivalantly,
the compact neighborhoods
of every point $x \in X$ form a neighborhood base at~$x.$

\begin{dfn}[Remark 2.5(vi) of~\cite{BP}]\label{D-TDZero311}
Let $X$ be a locally compact
but not necessarily Hausdorff topological space.
We say that $X$ has
{\emph{topological dimension zero}}
if for every $x \in X$ and every open set $U \subset X$
such that $x \in U,$
there exists a compact open (but not necessarily closed)
subset $Y \subset X$
such that $x \in Y \subset U.$
We further say that a C*-algebra~$A$
has
{\emph{topological dimension zero}}
if $\Prim (A)$ has topological dimension zero.
\end{dfn}

We do not write $\dim (X) = 0$ or use similar notation,
because other values of the topological dimension
are not defined for spaces of this generality.

\begin{lem}\label{R-604Her}
Let $A$ be a C*-algebra~$A$ with topological dimension zero.
Then every hereditary subalgebra $B$ of~$A$
also has topological dimension zero.
\end{lem}

\begin{proof}
Theorem 5.5.5 of~\cite{Mr}
implies that $\Prim (B)$
is homeomorphic to an open subset of $\Prim(A),$
and clearly an open subset of a space
with topological dimension zero
again has topological dimension zero.
\end{proof}

\begin{dfn}\label{D-COExh311}
Let $X$ be a not necessarily Hausdorff topological space.
A {\emph{compact open exhaustion}}
of~$X$
is an increasing net
$(Y_{\ld})_{\ld \in \Ld}$
of compact open subsets $Y_{\ld} \subset X$
such that $X = \bigcup_{\ld \in \Ld} Y_{\ld}.$
\end{dfn}

\begin{lem}\label{L-Union311}
Let $X$ be a not necessarily Hausdorff topological space.
The following are equivalent:
\begin{enumerate}
\item\label{L-Union311-COE}
$X$ has a compact open exhaustion.
\item\label{L-Union311-Nbd}
For every $x \in X$ there exists
a compact open (but not necessarily closed)
subset $Y \subset X$
such that $x \in Y.$
\end{enumerate}
\end{lem}

\begin{proof}
That (\ref{L-Union311-COE}) implies~(\ref{L-Union311-Nbd}) is obvious.
For the reverse,
for every $x \in X$ choose a compact open
subset $Z_x \subset X$
such that $x \in Z_x.$
Take $\Ld$ to be the collection of all finite subsets of~$X,$
and for $\ld \in \Ld$ set $Y_{\ld} = \bigcup_{x \in \ld} Z_x.$
\end{proof}

\begin{lem}\label{L-TDZCond311}
Let $X$ be a locally compact
but not necessarily Hausdorff topological space.
The following are equivalent:
\begin{enumerate}
\item\label{L-TDZCond311-TDZ}
$X$ has topological dimension zero.
\item\label{L-TDZCond311-COE}
Every open subset of $X$ has a compact open exhaustion.
\item\label{L-TDZCond311-Base}
The compact open sets form a base for the topology of~$X.$
\end{enumerate}
\end{lem}

\begin{proof}
The equivalence of (\ref{L-TDZCond311-TDZ}) and~(\ref{L-TDZCond311-COE})
follows from Lemma~\ref{L-Union311}.
The equivalence of
(\ref{L-TDZCond311-TDZ}) and~(\ref{L-TDZCond311-Base})
is the definition of a base for a topology.
\end{proof}
The proof of the following lemma is adapted
from the proof of Proposition~2.6 of~\cite{BP}.

\begin{lem}\label{L-Ext311}
Let $X$ be a locally compact
but not necessarily Hausdorff topological space.
Let $U \subset X$ be open.
Suppose that $X \setminus U$ has a compact open exhaustion
(in the relative topology)
and $U$ has a compact open exhaustion.
Then $X$ has a compact open exhaustion.
\end{lem}

\begin{proof}
We verify condition~(\ref{L-Union311-Nbd})
in Lemma~\ref{L-Union311}.
So let $x \in X.$
We need to find a compact open subset $Y \subset X$
such that $x \in Y.$

If $x \in U,$
then, since open subsets of $U$ are open in~$X,$
we simply apply the hypothesis on~$U.$
So assume that $x \in X \setminus U.$
By Lemma~\ref{L-Union311},
there is a compact subset $Z \subset X \setminus U$
such that $x \in Z$
and $Z$ is open in $X \setminus U$ in the relative topology.
Choose an open subset $W \subset X$
such that $Z = W \cap (X \setminus U).$
Since $X$ is locally compact,
for every $z \in Z$ there is a compact set $R (z) \subset X$
such that
\[
z \in \sint (R (z)) \subset R (z) \subset W.
\]
Since $Z$ is compact,
there are $n \in \N$ and $z_1, z_2, \ldots, z_n \in Z$
such that
\[
\sint (R (z_1)), \, \sint (R (z_2)), \, \ldots, \, \sint (R (z_n))
\]
cover~$Z.$
Set $R = \bigcup_{k = 1}^n R (z_k).$
Then $R$ is compact and $Z \subset \sint (R) \subset R \subset W.$
Since $W \cap (X \setminus U) = Z,$
it follows that $R \setminus \sint (R) \subset U.$
Also, $R \setminus \sint (R)$
is closed in~$R$ in the relative topology of~$R,$
and is hence compact.
Since $U$ has a compact open exhaustion,
there is a compact open subset $C \subset U$
such that $R \setminus \sint (R) \subset C.$
Set $Y = C \cup R.$
Then $Y$ is compact because $C$ and $R$ are.
Also, $Y$ is open because $Y = C \cup \sint (R).$
Since $x \in R \subset Y,$ the proof is complete.
\end{proof}

\begin{lem}\label{L-COESm321}
Let $X$ be a topological space with a compact open exhaustion.
Then every closed subset of $X$ has a compact open exhaustion
in the relative topology.
\end{lem}

\begin{proof}
Let $Y \subset X$ be closed.
Choose an increasing net
$(L_{\ld})_{\ld \in \Ld}$
of compact open subsets $L_{\ld} \subset X$
such that $X = \bigcup_{\ld \in \Ld} L_{\ld}.$
For $\ld \in \Ld,$
set $M_{\ld} = L_{\ld} \cap Y.$
Then $M_{\ld}$ is compact because $Y$ is closed
and, even when the space is not Hausdorff,
the intersection of a compact set and a closed set is compact.
Clearly $\bigcup_{\ld \in \Ld} M_{\ld} = Y.$
\end{proof}

\begin{dfn}[See page~53 of~\cite{PR}]\label{D:D1_0123}
An ideal $I$ in a C*-algebra $A$ is said to be {\emph{compact}}
if whenever $(I_{\ld})_{\ld \in \Ld}$
is an increasing net of ideals in~$A$
such that $I = {\overline{\bigcup_{\ld \in \Ld} I_{\ld} }},$
then there is $\ld \in \Ld$ such that $I = I_{\ld}.$
\end{dfn}

\begin{lem}\label{L-CptId311}
Let $A$ be a C*-algebra and let $I \subset A$ be an ideal.
Then $I$ is compact if and only if
$\Prim (I)$ is a compact open subset of $\Prim (A).$
\end{lem}

\begin{proof}
Compactness of $I$
is equivalent to the statement that
whenever $(U_{\ld})_{\ld \in \Ld}$
is an increasing net of open subsets of $\Prim (A)$
such that $\Prim (I) = \bigcup_{\ld \in \Ld} U_{\ld},$
then there is $\ld \in \Ld$ such that $\Prim (I) = U_{\ld}.$
\end{proof}

\begin{lem}\label{L-Inter322}
Let $\af \colon G \to \Aut (A)$
be an action of a finite group $G$
on a C*-algebra~$A.$
Let $J$ be an $\alpha$-invariant ideal in~$A.$
Then $J \subset {\overline{A (J \cap A^{\af})}}.$
\end{lem}

\begin{proof}
Set $n = \card (G).$
Choose an approximate identity $(a_{\ld})_{\ld \in \Ld}$ for~$J.$
For $\ld \in \Ld,$ define
\[
b_{\ld} = \frac{1}{n} \sum_{g \in G} \af_g (a_{\ld}).
\]
Since $J$ is $G$-invariant,
$(\af_g (a_{\ld}))_{\ld \in \Ld}$
is also an approximate identity for~$J$
for every $g \in G.$
So $(b_{\ld})_{\ld \in \Ld}$
is an approximate identity for~$J.$
Moreover, $b_{\ld} \in A^{\af}$ for all $\ld \in \Ld.$
If $x \in J,$
we therefore have
$x = \lim_{\ld \in \Ld} x b_{\ld} \in {\overline{A (J \cap A^{\af})}}.$
This completes the proof.
\end{proof}

\begin{lem}\label{L:L1_0123}
Let $\af \colon G \to \Aut (A)$
be an action of a finite group $G$
on a C*-algebra~$A.$
Let $J$ be an $\alpha$-invariant ideal in~$A.$
If $J \cap A^{\af}$ is a compact ideal in~$A^{\af},$
then $J$ is a compact ideal in~$A.$
\end{lem}

\begin{proof}
Set $I = J \cap A^{\af}.$
Let $(J_{\ld})_{\ld \in \Ld}$ be an increasing net of ideals
in $A$ such that
$J = {\overline{ \bigcup_{\ld \in \Ld} J_{\ld} }}.$
We must find $\ld$ such that $J_{\ld} = J.$
For $\ld \in \Ld,$ define $I_{\ld} = J_{\ld} \cap A^{\af},$
which is an ideal in~$A^{\af}.$

We claim that
${\overline{ \bigcup_{\ld \in \Ld} I_{\ld} }} = I.$
To prove this, let $a \in I$ and let $\ep > 0.$
For $g \in G,$
we have
\[
a = \af_g (a)
  \in \af_g (J)
  = {\overline{ \bigcup_{\ld \in \Ld} \af_g (J_{\ld}) }},
\]
so there are
$\ld_g \in \Ld$
and
$x_g \in \af_g ( J_{\ld_{g}} )$
such that
$0 \leq x_{g} \leq 1$
and
% $\| a x_{g} - a \| < \frac{\ep}{n}.$
$\| a x_{g} - a \| < n^{-1} \ep.$
Choose $\ld \in \Ld$ such that for all $g \in G$
we have $\ld \geq \ld_{g}.$

Choose a bijection
$k \colon \{ 1, 2, \ldots, n \} \to G.$
Define
\[
x = \frac{1}{n}
   \sum_{g \in G} \af_{g} (x_{k (1)} x_{k (2)} \cdots x_{k (n)} ).
\]
We have
\[
x_{k (1)} x_{k (2)} \cdots x_{k (n)}
  \in \bigcap_{g \in G} \af_{g} ( J_{\ld} ).
\]
Since $\bigcap_{g \in G} \af_{g} ( J_{\ld} )$ is $\af$-invariant,
it follows that
$x \in \bigcap_{g \in G} \af_{g} ( J_{\ld} ).$
In particular, $x \in J_{\ld}.$
Clearly $x \in A^{\af},$
so $x \in I_{\ld}.$
Also,
% \[
% \| a x_{k (1)} - a \| < \frac{\ep}{n},
% \]
$\| a x_{k (1)} - a \| < n^{-1} \ep,$
so
\[
\| a x_{k (1)} x_{k (2)} - a \|
  \leq \| a x_{k (1)} - a \| \cdot \| x_{k (2)} \|
          + \| a x_{k (2)} - a \|
   < \frac{\ep}{n} + \frac{\ep}{n}
   = \frac{2 \ep}{n}.
\]
An induction argument shows that
\[
\big\| a x_{k (1)} x_{k (2)} \cdots x_{k (n)} - a \big\|
   < \ep.
\]
Since $\af_{g}^{-1} (a) = a$
for all $g \in G,$
we also have
\[
\big\| a \cdot
  \af_{g} (x_{k (1)} x_{k (2)} \cdots x_{k (n)} ) - a \big\| < \ep
\]
for all $g \in G.$
Therefore $\| a x - a \| < \ep.$
Since $a x \in I_{\ld},$
we have $\dist (a, I_{\ld}) < \ep.$
Since $\ep > 0$ is arbitrary,
it follows that $a \in {\overline{ \bigcup_{\ld \in \Ld} I_{\ld} }},$
proving the claim.

Since $I$ is compact, there is $\ld \in \Ld$
such that $I = I_{\ld}.$
Using Lemma~\ref{L-Inter322},
we then get
$J \subset {\overline{A I}}
  = {\overline{A I_{\ld}}}
  \subset J_{\ld}.$
Thus $J_{\ld} = J,$
showing that $J$ is compact.
\end{proof}

% ???
% (The following remark should be deleted in the final version.)
%
% There is no reason for the element $x$ in the proof to be positive,
% or even selfadjoint.
% If we start instead by letting
% $S$ be the set of all bijections from
% $\{ 1, 2, \ldots, n \}$ to~$G,$
% and then averaging
% \[
% \frac{1}{n!}
%    \sum_{\sm \in S} x_{\sm (1)} x_{\sm (2)} \cdots x_{\sm (n)}
% \]
% over~$G,$
% we will at least get something selfadjoint.
% (The adjoint of any term is equal to one of the other terms.)
% But the result still won't obviously be positive.

\begin{lem}\label{L-COE322}
Let $\af \colon G \to \Aut (A)$
be an action of a finite group $G$
on a C*-algebra~$A.$
Suppose that whenever $I \subset A^{\af}$ is a compact ideal,
then ${\overline{A I A}} \cap A^{\af}$ is a compact ideal in~$A^{\af}.$
Let $J$ be an $\alpha$-invariant ideal in~$A.$
If $\Prim (J \cap A^{\af}) \subset \Prim (A^{\af})$
has a compact open exhaustion,
then $\Prim (J) \subset \Prim (A)$
has a compact open exhaustion.
\end{lem}

\begin{proof}
By Lemma~\ref{L-CptId311},
it suffices to find an increasing net
$(J_{\ld})_{\ld \in \Ld}$
of compact ideals in~$A$
such that $J = {\overline{\bigcup_{\ld \in \Ld} J_{\ld} }}.$
Lemma~\ref{L-CptId311} implies that
there is an increasing net
$(I_{\ld})_{\ld \in \Ld}$
of compact ideals in~$A^{\af}$
such that
$J \cap A^{\af} = {\overline{\bigcup_{\ld \in \Ld} I_{\ld} }}.$
For $\ld \in \Ld,$ set $J_{\ld} = {\overline{A I_{\ld} A}}.$
By hypothesis, ${\overline{A I_{\ld} A}} \cap A^{\af}$ is compact.
So Lemma~\ref{L:L1_0123} implies that $J_{\ld}$ is compact.
Clearly $J_{\ld} \subset J.$
Since $J$ is $\af$-invariant,
Lemma~\ref{L-Inter322} implies the first step of the
following calculation:
\[
J \subset {\overline{A (J \cap A^{\af})}}
  \subset {\overline{A (J \cap A^{\af}) A}}
  \subset {\overline{A \left( \bigcup_{\ld \in \Ld} I_{\ld} \right) A}}
  = {\overline{\bigcup_{\ld \in \Ld} A I_{\ld} A}}
  = {\overline{\bigcup_{\ld \in \Ld} J_{\ld} }}.
\]
This completes the proof.
\end{proof}

\begin{thm}\label{T-AG-321}
Let $\af \colon G \to \Aut (A)$
be an action of a finite group $G$
on a C*-algebra~$A.$
Suppose that $A^{\af}$ has topological dimension zero.
Suppose also that whenever $I \subset A^{\af}$ is a compact ideal,
then ${\overline{A I A}} \cap A^{\af}$ is a compact ideal
in~$A^{\af}.$
Then $A$ has topological dimension zero.
\end{thm}

\begin{proof}
We verify that $\Prim( A )$
satisfies condition~(\ref{L-TDZCond311-COE})
of Lemma~\ref{L-TDZCond311}.
So let $V \subset \Prim( A )$ be open.
Let $I \subset A$ be the corresponding ideal.
We will use Lemma 5.3.3 of~\cite{Ph1}.
Thus, for $S \subset G$ with $1 \in S,$
we define $\af$-invariant ideals $I_S, I_S^{-} \subset A$ by
\[
I_S = \sum_{g \in G} \af_g \left( \bigcap_{h \in S} \af_h (I) \right)
\andeqn
I_S^{-} = \sum_{g \in G \setminus S} I_{S \cup \{ g \} }.
\]
Then
$I_G \subset I \subset I_{ \{ 1 \} }$
(Lemma 5.3.3(3) of~\cite{Ph1}).

We prove by downwards induction on the set~$S$
that $\Prim (I_S \cap I)$ has a compact open exhaustion.
We start with $S = G.$
Since $I_G \cap I = I_G$ is $\af$-invariant,
this follows from Lemma~\ref{L-COE322}.

Suppose now $S$ is given, with $1 \in S,$
and we know that $\Prim (I_T \cap I)$ has a compact open exhaustion
for all $T \subset G$ such that $S \subset T$ and $S \neq T.$
We have
\[
I_S^{-} \cap I
 =  \sum_{g \in G \setminus S} (I_{S \cup \{ g \} } \cap I),
\]
so
\[
\Prim (I_S^{-} \cap I)
  = \bigcup_{g \in G \setminus S} \Prim (I_{S \cup \{ g \} } \cap I).
\]
It is easily checked that the union of open subsets
with compact open exhaustions
again has a compact open exhaustion.
So $\Prim (I_S^{-} \cap I)$ has a compact open exhaustion.

Since $I_S$ is $G$-invariant,
Lemma~\ref{L-COE322}
implies that $\Prim (I_S)$ has a compact open exhaustion,
and now Lemma~\ref{L-COESm321}
implies that $\Prim (I_S / I_S^{-})$ has a compact open exhaustion.
Lemma 5.3.3(5) of~\cite{Ph1} implies that
$[(I_S \cap I) + I_S^{-}] / I_S^{-}$
is a direct summand in $I_S / I_S^{-}.$
Therefore $\Prim \big( [(I_S \cap I) + I_S^{-}] / I_S^{-} \big)$
has a compact open exhaustion by Lemma~\ref{L-COESm321}.
We have a short exact sequence
\[
0 \longrightarrow I_S^{-} \cap I
  \longrightarrow I_S \cap I
  \longrightarrow [(I_S \cap I) + I_S^{-}] / I_S^{-}
  \longrightarrow 0,
\]
so Lemma~\ref{L-Ext311}
implies that $\Prim (I_S \cap I)$ has a compact open exhaustion.
This completes the induction step.

Since $I_{ \{ 1 \} } \cap I = I,$
we conclude that $V = \Prim (I)$ has a compact open exhaustion.
\end{proof}

The hypothesis in Theorem~\ref{T-AG-321}
involving ${\overline{A I A}} \cap A^{\af}$
is annoying.
(It enters via Lemma~\ref{L-COE322}.)

\begin{qst}\label{Q-CptHyp322}
Let $\af \colon G \to \Aut (A)$
be an action of a finite group~$G$
on a C*-algebra~$A.$
Suppose that $I \subset A^{\af}$ is a compact ideal.
Does it follow that ${\overline{A I A}} \cap A^{\af}$
is also a compact ideal
in~$A^{\af}$?
\end{qst}

We note that ${\overline{A I A}} \cap A^{\af}$ need not equal~$I.$
Indeed,
there are many cases in which $A$ is simple but $A^{\af}$ is not.
(For example, take $A = L (l^2 (G))$ with the action
given by conjugation by the regular representation of~$G.$)
Then for any nontrivial ideal $I \subset A^{\af},$
one gets ${\overline{A I A}} \cap A^{\af} = A^{\af} \neq I.$

If $A$ is a crossed product of $A^{\af}$ by an action
of an abelian group,
then Question~\ref{Q-CptHyp322} has a positive answer.
Accordingly, we get a result for abelian groups.
To keep the notation simple,
we start with a separate lemma.

\begin{lem}\label{L-ClACap-322}
Let $\af \colon G \to \Aut (A)$
be an action of a finite group~$G$
on a C*-algebra~$A.$
Let $I \subset A$ be an ideal.
Then
\[
{\overline{C^* (G, A, \alpha) I C^* (G, A, \alpha) }} \cap A
  = \sum_{g \in G} \af_g (I).
\]
\end{lem}

\begin{proof}
For $g \in G,$
let $u_g \in M (C^* (G, A, \alpha))$
be the standard unitary corresponding to~$g,$
as in the introduction.
Then
\[
\af_g (I) = u_g I u_g^*
      \subset
         {\overline{M (C^* (G, A, \alpha)) I M (C^* (G, A, \alpha)) }}
      = {\overline{C^* (G, A, \alpha) I C^* (G, A, \alpha) }}.
\]
This proves one of the desired inclusions.

For the other,
set $L = \sum_{g \in G} \af_g (I).$
Then $L$ is $\af$-invariant,
so $C^* (G, L, \alpha)$ is an ideal in $C^* (G, A, \alpha)$
which clearly contains~$I.$
Also $C^* (G, L, \alpha) \cap A = L.$
Therefore
\[
{\overline{C^* (G, A, \alpha) I C^* (G, A, \alpha) }} \cap A
 \subset C^* (G, L, \alpha) \cap A
 = L
 = \sum_{g \in G} \af_g (I).
\]
This completes the proof.
\end{proof}

\begin{thm}\label{T-AbCP321}
Let $\af \colon G \to \Aut (A)$
be an action of a finite abelian group~$G$
on a C*-algebra~$A.$
Suppose that $A$ has topological dimension zero.
Then $C^* (G, A, \af)$ and $A^{\af}$ have topological dimension zero.
\end{thm}

\begin{proof}
We first consider $C^* (G, A, \alpha).$
Let ${\widehat{\af}} \colon
  {\widehat{G}} \to \Aut \big( C^* (G, A, \alpha) \big)$
be the dual action.
Then
$C^* (G, A, \alpha)^{\widehat{\af}} \cong A,$
so has topological dimension zero.

We check that ${\widehat{\af}}$
satisfies the second hypothesis of Theorem~\ref{T-AG-321}.
Let $I \subset A$ be a compact ideal.
Applying Lemma~\ref{L-ClACap-322}
we get
\[
{\overline{C^* (G, A, \alpha) I C^* (G, A, \alpha) }}
   \cap A
 = \sum_{g \in G} \af_{g} (I).
\]
The right hand side is compact by Lemma~\ref{L-CptId311}
and because the union of finitely many compact sets is compact.

Theorem~\ref{T-AG-321} now implies that
$C^* (G, A, \alpha)$ has topological dimension zero.

The result for~$A^{\af}$ now follows from
Lemma~\ref{R-604Her}
and Proposition~\ref{P-FP303}.
\end{proof}

We now turn to an action~$\af$ of a general finite group~$G$
on a C*-algebra~$A,$
but under the assumption
that $A$ separates the ideals in $C^*(G, A, \alpha).$
The proof is much more straightforward,
but the machinery developed above seems to be of little help.

\begin{prp}\label{P:P1_0123}
Let $\af \colon G \to \Aut (A)$ be an action of a finite group $G$
on a C*-algebra~$A.$
Assume that $A$ separates the ideals in $C^*(G, A, \alpha).$
Suppose that $A$ has topological dimension zero.
Then $C^* (G, A, \alpha)$ and~$A^{\af}$
have topological dimension zero.
\end{prp}

Part of the proof works for reduced crossed products
by actions of discrete groups,
so we give it in that generality.
Let $\af \colon G \to \Aut (A)$ be an action of a discrete group $G$
on a C*-algebra~$A.$
Let $E \colon C^*_{\mathrm{r}} (G, A, \alpha) \to A$
be the canonical conditional expectation
(as in the introduction).
It is immediate that if $I \subset A$ is an $\alpha$-invariant ideal,
then
\begin{equation}\label{Eq:CIE-304}
E (C^*_{\mathrm{r}} (G,  I,  \alpha)) = I.
\end{equation}
It follows that for $\alpha$-invariant ideals $I_1, I_2 \subset A,$
we have
\begin{equation}\label{Eq:Incl}
{\mbox{$I_1 \subset I_2$ if and only if
$C^*_{\mathrm{r}} (G, I_1, \alpha)
 \subset C^*_{\mathrm{r}} (G, I_2, \alpha).$}}
\end{equation}

\begin{lem}\label{L:L1_0123-2}
Let $\af \colon G \to \Aut (A)$ be an action of a discrete group $G$
on a C*-algebra~$A.$
Suppose $A$ separates the ideals in $C^*_{\mathrm{r}} (G, A, \alpha).$
Let $I$ be an $\alpha$-invariant ideal of~$A.$
If $I$ is compact, then $C^*_{\mathrm{r}} (G, I, \alpha)$
is also compact.
\end{lem}

\begin{proof}
Let $(J_{\ld})_{\ld \in \Ld}$ be an increasing net of ideals
in $C^*_{\mathrm{r}} (G, A, \alpha)$ such that
\[
C^*_{\mathrm{r}} (G, I, \alpha)
 = {\overline{ \bigcup_{\ld \in \Ld} J_{\ld} }}.
\]
By hypothesis,
there are $\alpha$-invariant ideals $I_{\ld}$
such that $J_{\ld} = C^*_{\mathrm{r}} (G, I_{\ld}, \alpha)$
for all $\ld \in \Ld.$
By~(\ref{Eq:Incl}),
we have $I_{\ld} \subset I$ for all $\ld \in \Ld,$
and moreover $(I_{\ld})_{\ld \in \Ld}$ is increasing.
By~(\ref{Eq:CIE-304}) and because $E$ is continuous, we have
\begin{align*}
I & = E (C^*_{\mathrm{r}} (G, I, \alpha))
    = E \left( {\overline{\bigcup_{\ld \in \Ld}
            C^*_{\mathrm{r}}(G, I_{\ld}, \alpha) }} \right)
     \\
  & \subset {\overline{ E \left(
       \bigcup_{\ld \in \Ld}
          C^*_{\mathrm{r}} (G, I_{\ld}, \alpha) \right) }}
    = {\overline{ \bigcup_{\ld \in \Ld}
          E (C^*_{\mathrm{r}} (G, I_{\ld},  \alpha)) }}
    = {\overline{ \bigcup_{\ld \in \Ld} I_{\ld} }}
    \subset I.
\end{align*}
Thus $I = {\overline{\bigcup_{\ld \in \Ld}  I_{\ld} }}.$
Since $I$ is compact, there is $\ld \in \Ld$
such that $I = I_{\ld}.$
Then
$C^*_{\mathrm{r}} (G, I, \alpha)
 = C^*_{\mathrm{r}} (G, I_{\ld}, \alpha) = J_{\ld}.$
This shows that $C^*_{\mathrm{r}} (G, I, \alpha)$ is compact.
\end{proof}

\begin{proof}[Proof of Proposition~\ref{P:P1_0123}]
We first consider $C^* (G, A, \alpha).$

It follows from Lemma~\ref{L-TDZCond311}
and Lemma~\ref{L-CptId311}
that a C*-algebra~$D$ has topological dimension zero
if and only if every ideal in $D$ is the closure of the union of an
increasing net of compact ideals.

So let $J$ be an arbitrary ideal in $C^* (G, A, \alpha).$
By hypothesis, there is an $\alpha$-invariant ideal $I \subset A$
such that $J = C^* (G, I, \alpha).$
Since $A$ has topological dimension zero, there is an increasing net
$(I_{\ld})_{\ld \in \Ld}$ of compact ideals of $A$ such that
$I = {\overline{\bigcup_{\ld \in \Ld} I_{\ld} }}.$
For $\ld \in \Ld,$ define $L_{\ld} = \sum_{g \in G} \alpha_g (I_{\ld}).$
Then $(L_{\ld})_{\ld \in \Ld}$ is an increasing net
of $\alpha$-invariant ideals
and ${\overline{\bigcup_{\ld \in \Ld} L_{\ld} }} = I.$
Since a finite union of compact sets is compact,
Lemma~\ref{L-CptId311} implies that
$L_{\ld}$ is compact for all $\ld \in \Ld.$
The ideals
$C^* ( G, L_{\ld}, \af )$
are compact by Lemma~\ref{L:L1_0123-2}.
By~(\ref{Eq:Incl}),
these ideals are increasing and satisfy
\[
{\overline{\bigcup_{\ld \in \Ld}
   C^* ( G, L_{\ld}, \af ) }}
 = C^* (G, I, \alpha).
\]
This completes the proof for $C^* (G, A, \alpha).$

The result for~$A^{\af}$ now follows from
Lemma~\ref{R-604Her}
and Proposition~\ref{P-FP303}.
\end{proof}

\begin{cor}\label{C-2729TdzPtOut}
Let $\af \colon G \to \Aut (A)$ be a strongly pointwise outer action
of a finite group $G$ on a C*-algebra~$A.$
Suppose that $A$ has topological dimension zero.
Then $C^* (G, A, \alpha)$ and~$A^{\af}$
have topological dimension zero.
\end{cor}

\begin{proof}
It follows from Theorem~\ref{P:IdealPrsv}
that $A$ separates the ideals in $C^* (G, A, \af).$
Now apply Proposition~\ref{P:P1_0123}.
\end{proof}

\section{Purely infinite C*-algebras with finite primitive
 spectrum}\label{Sec:0123}

Let $A$ be a purely infinite C*-algebra,
and let $\af \colon G \to \Aut (A)$ be an action of a finite group~$G$
on~$A.$
We know of no examples in which
$C^* (G, A, \alpha)$ and $A^{\af}$ are not purely infinite,
but we have not been able to prove that they are.
The main result of this section is that,
if $A$ has a composition series in which all the subquotients
have finite primitive ideal spaces,
then $C^* (G, A, \alpha)$ and $A^{\af}$ must be purely infinite.
A counterexample to the general statement,
if it exists, must therefore be fairly complicated.

We record the following standard fact.

\begin{lem}\label{L-FinIdeals304}
Let $A$ be a C*-algebra.
Then $A$ has finitely many ideals if and only if $\Prim (A)$ is finite.
\end{lem}

\begin{proof}
The forward implication is trivial.
The reverse follows from the fact that ideals in $A$
are in one to one correspondence with open subsets of $\Prim (A).$
\end{proof}

\begin{lem}\label{L:L1_0126}
Let $A$ be a C*-algebra and let $I$ be an ideal of~$A.$
The following are equivalent:
\begin{enumerate}
\item\label{L:L1_0126:1}
$\Prim (A)$ is finite.
\item\label{L:L1_0126:2}
$\Prim (I)$ and $\Prim (A / I)$ are finite.
\end{enumerate}
\end{lem}

\begin{proof}
The result is immediate from the fact that
$\Prim (A)$ is the (not necessarily topological)
disjoint union of $\Prim (I)$ and $\Prim (A/I).$
\end{proof}

\begin{cor}\label{C-OldL2_0123}
Let $I$ and $J$ be ideals in a C*-algebra~$A$
which have finite primitive ideal spaces.
Then $\Prim (I + J)$ is finite.
\end{cor}

\begin{proof}
Consider the short exact sequence of C*-algebras
\[
0 \longrightarrow I
  \longrightarrow I + J
  \longrightarrow J / (I \cap J)
  \longrightarrow 0.
\]
The space $\Prim ( J / (I \cap J) )$ is finite
by one direction in Lemma~\ref{L:L1_0126},
so $\Prim (I + J)$ is finite by the other direction.
\end{proof}

\begin{lem}\label{L:L2_0123}
Let $I$ and $J$ be ideals in a C*-algebra~$A.$
Assume that $I$ and $J$ are purely infinite.
Then $I + J$ is purely infinite.
\end{lem}

\begin{proof}
Consider the short exact sequence of C*-algebras
\[
0 \longrightarrow I
  \longrightarrow I + J
  \longrightarrow J / (I \cap J)
  \longrightarrow 0.
\]
Pure infiniteness
passes to quotients (Proposition~4.3 of~\cite{KR}),
so $J / (I \cap J)$ is purely infinite.
Extensions of purely infinite C*-algebras are purely infinite
(Theorem~4.19 of~\cite{KR}),
so $I + J$ is purely infinite.
\end{proof}

\begin{lem}\label{L:L2_0126}
Let $\af \colon G \to \Aut (A)$ be an action of a finite group $G$
on a C*-algebra~$A.$
Suppose that $A$ is the sum of orthogonal ideals which
are permuted transitively by the action of~$G.$
Let $I$ be one of these ideals,
and let $H$ be the subgroup of elements of $G$
which carry $I$ into itself, so that $H$ acts on~$I.$
Set $n = \card (G / H).$
Then $C^* (G, A, \af) \cong M_n (C^* (H, I, \alpha)).$
\end{lem}

\begin{proof}
As described after the proof of Proposition~2.3 of~\cite{RfFG},
this is a special case of Theorem 2.13(i) of~\cite{Gr}.
\end{proof}

\begin{lem}\label{L:L3_0126}
Let $\af \colon G \to \Aut (A)$ be an action of a finite group $G$
on a C*-algebra~$A.$
Assume that $A$ is $\af$-simple.
Then $C^* (G, A, \af)$
is a finite direct sum of simple C*-algebras.
Moreover,
if in addition $A$ is purely infinite,
then $C^* (G, A, \alpha)$ is a finite direct sum
of simple purely infinite C*-algebras.
\end{lem}

\begin{proof}
Since $A$ is $\alpha$-simple,
the discussion at the beginning of Section~3 of~\cite{RfFG}
implies that the hypotheses of Lemma~\ref{L:L2_0126} hold,
with, in addition, $I$~being simple.
Set $n = \card (G / H).$
Then $C^* (G, A, \af) \cong M_n (C^* (H, I, \alpha)).$
Theorem 3.1 of~\cite{RfFG} now implies that $C^* (G, A, \af)$ is a
finite direct of simple C*-algebras.

If $I$ is unital,
the second part conclusion now follows from Theorem~4.5 of~\cite{JO}.
However, the proof of this theorem
(including the proofs of Theorem~4.2
and Corollary~4.4 of~\cite{JO};
the nonunital case of Theorem~2.1 of~\cite{JO}
is the first part of the present lemma)
also works for nonunital C*-algebras,
provided one uses multiplier algebras at the appropriate places.
\end{proof}

\begin{thm}\label{L:L3_0114}
Let $\af \colon G \to \Aut (A)$ be an action of a finite group $G$
on a C*-algebra~$A.$
Assume that $A$ has finitely many $\alpha$-invariant ideals.
Then $\Prim (C^* (G, A, \alpha))$ is finite.
Moreover, if in addition $A$ is purely infinite,
then $C^* (G, A, \alpha)$ is purely infinite.
\end{thm}

\begin{proof}
Since $A$ has finitely many $\alpha$-invariant ideals,
every $\alpha$-invariant ideal $I \subset A$
contains a maximal strictly smaller $\alpha$-invariant ideal~$J.$
In particular, $I / J$ is $\alpha$-simple.
An induction argument therefore shows that
$A$ has a finite composition series
$(I_k)_{0 \leq k \leq n}$
consisting of $\alpha$-invariant ideals
such that $I_{k + 1} / I_k$ is $\alpha$-simple
for $k = 0, 1, \ldots, n - 1.$

Since $G$ is amenable, for $k = 0, 1, \ldots, n - 1$ we have a
short exact sequence of C*-algebras
\begin{equation}\label{Eq:Star}
0 \longrightarrow C^* (G, I_k, \alpha)
  \longrightarrow C^* (G, I_{k + 1}, \alpha)
  \longrightarrow C^* (G, \, I_{k + 1} / I_k, \, \alpha)
  \longrightarrow 0.
\end{equation}
Lemma~\ref{L:L3_0126} implies that
$C^* (G, \, I_{k + 1} / I_k, \, \alpha)$
has finite primitive ideal space.
Moreover, if $A$ is purely infinite,
then $I_{k + 1} / I_k$ is purely infinite
by Theorem~4.19 of~\cite{KR},
so Lemma~\ref{L:L3_0126} implies that
$C^* (G, \, I_{k + 1} / I_k, \, \alpha)$ is also purely infinite.

An induction argument using Lemma~\ref{L:L1_0126}
now shows that $\Prim (C^* (G, A, \alpha))$ is finite.
If $A$ is purely infinite,
using Theorem~4.19 of~\cite{KR} in the induction argument
shows that then $C^* (G, A, \alpha)$ is also purely infinite.
\end{proof}

\begin{cor}\label{L:L5_0126}
Let $\af \colon G \to \Aut (A)$ be an action of a finite group $G$
on a C*-algebra~$A.$
If $A$ is purely infinite
and has finitely many $\af$-invariant ideals,
then $C^* (G, A, \af)$ has the ideal property.
\end{cor}

\begin{proof}
By Theorem~\ref{L:L3_0114},
the algebra $C^* (G, A, \af)$ is purely infinite
and has finite primitive ideal space.
Proposition~2.3 of~\cite{KNP}
therefore implies that $C^* (G, A, \af)$ has the ideal property.
\end{proof}

\begin{cor}\label{C-L3_0114-304}
Let $\af \colon G \to \Aut (A)$ be an action of a finite group $G$
on a C*-algebra~$A.$
Assume that $A$ has finitely many $\alpha$-invariant ideals.
Then $\Prim (A^{\af})$ is finite.
Moreover, if in addition $A$ is purely infinite,
then $A^{\af}$ is purely infinite.
\end{cor}

\begin{proof}
Theorem~\ref{L:L3_0114} gives the conclusions for $C^* (G, A, \af).$
Now apply Proposition~\ref{P-FP303}.
The first conclusion follows from the fact that
the primitive ideal space of a corner
is a subset of the primitive ideal space of the original algebra.
The second conclusion follows from the fact
(Proposition~4.17 of~\cite{KR})
that hereditary subalgebras of purely infinite algebras
are again purely infinite.
\end{proof}

\begin{rmk}\label{R-Exs}
There are a fair number of C*-algebras which are purely infinite
and have finite primitive ideal spaces,
as required in
Theorem~\ref{L:L3_0114} and Corollary~\ref{C-L3_0114-304}.

For example,
this is true of the C*-algebras of many finite graphs.
If a finite graph~$E$
satisfies Condition~(K) as defined before Examples~4.6 of~\cite{Rb},
then $\Prim (C^* (E))$ is finite.
Necessary and sufficient conditions for pure infiniteness of $C^* (E)$
are given in Theorem~2.3 of~\cite{HS}.
(We are grateful to Gene Abrams for help locating this reference.)
In particular, the conditions in parts (d) and~(e)
of that theorem are entirely in terms of the graph.
(Warning: the convention relating the direction
of the arrows in a graph~$E$ to the definition of $C^* (E)$
in~\cite{HS} is the opposite of the convention in~\cite{Rb}.)

Theorem~\ref{L:L3_0114} and Corollary~\ref{C-L3_0114-304}
also apply to many minimal tensor products
$B \otimes_{\mathrm{min}} D$ in which $B$ is purely infinite simple
and $\Prim (D)$ is finite.
For example, it is enough to assume that $B$ is exact
(which ensures that $\Prim (B \otimes_{\mathrm{min}} D) \cong \Prim (D)$
when $\Prim (D)$ is finite)
and approximately divisible
(so that Theorem~4.5 of~\cite{KR} applies).
\end{rmk}

We can get a little farther by considering composition series.

\begin{thm}\label{T-CompSer-304}
Let $A$ be a purely infinite C*-algebra.
Suppose there is an ordinal~$\kp$
and a composition series $(I_{\ld})_{\ld \leq \kp}$ for~$A$
such that $\Prim (I_{\ld + 1} / I_{\ld})$ is finite for all $\ld < \kp.$
Let $G$ be a finite group,
and let $\af \colon G \to \Aut (A)$ be any action of $G$ on~$A.$
Then $C^* (G, A, \alpha)$ and $A^{\af}$ are purely infinite
and have composition series
in which all the subquotients have finite primitive ideal spaces.
\end{thm}

\begin{proof}
We first claim that $A$ has a $G$-invariant composition series
$(J_{\ld})_{\ld \leq \kp}$
such that $\Prim (J_{\ld + 1} / J_{\ld})$ is finite for all $\ld < \kp.$
For use in this argument,
for an ideal $I$ in a C*-algebra~$B$
with an action $\bt \colon G \to \Aut (B),$
we write $G I$ for the ideal
$\sum_{g \in G} \bt_g (I).$
Define $J_{\ld} = G I_{\ld}$ for $\ld \leq \kp.$
It is easy to check that $(J_{\ld})_{\ld \leq \kp}$
is a $G$-invariant composition series.
(Some of the subquotients may be zero, but this does not matter.)
So we must show that, for $\ld < \kp,$
the space $\Prim \big( (G I_{\ld + 1}) / (G I_{\ld}) \big)$ is finite.
First, we have
\[
( G I_{\ld} + I_{\ld + 1} ) / (G I_{\ld})
  \cong I_{\ld + 1} / (I_{\ld + 1} \cap G I_{\ld}).
\]
Applying Lemma~\ref{L:L1_0126}
with $I_{\ld + 1} / I_{\lambda}$ in place of~$A,$
we see that this algebra has finite primitive ideal space.
% Let ${\overline{\af}} \colon G \to \Aut ( A / (G I_{\ld}) )$
% denote the action on the quotient.
As ideals in $A / (G I_{\ld}),$
we have
\[
(G I_{\ld + 1}) / (G I_{\ld})
  = G \big( ( G I_{\ld} + I_{\ld + 1} ) / (G I_{\ld}) \big)
  = \sum_{g \in G} \af_g
     \big( ( G I_{\ld} + I_{\ld + 1} ) / (G I_{\ld}) \big),
\]
which has finite primitive ideal space by Corollary~\ref{C-OldL2_0123}.
This completes the proof of the claim.

It now follows from Theorem~\ref{L:L3_0114}
that $C^* (G, A, \alpha)$ has a composition series
in which all the subquotients have finite primitive ideal spaces.
It easily follows from Proposition~\ref{P-FP303}
that the same is true for $A^{\af}.$

We now prove by induction on~$\ld$
that $C^* (G, J_{\ld}, \af)$ is purely infinite,
starting with the smallest $\ld_0$ such that $J_{\ld_0} \neq \{ 0 \}.$
Theorem~4.19 of~\cite{KR} implies
that $J_{\ld_0}$ is purely infinite,
and by construction $\Prim ( J_{\ld_0} )$ is finite,
so $C^* (G, J_{\ld_0}, \af)$ is purely infinite
by Theorem~\ref{L:L3_0114}.

Suppose now that $C^* (G, J_{\mu}, \af)$ is known to be purely infinite
for all $\mu < \ld.$
If there is $\mu$ such that $\mu + 1 = \ld,$
then we have a short exact sequence
\[
0 \longrightarrow C^* (G, J_{\mu}, \alpha)
  \longrightarrow C^* (G, \, J_{\mu + 1}, \, \alpha)
  \longrightarrow
     C^* (G, \, J_{\mu + 1} / J_{\mu}, \, \af)
  \longrightarrow 0,
\]
in which $C^* (G, J_{\mu}, \alpha)$ is purely infinite
by the induction hypothesis
and $C^* (G, \, J_{\mu + 1} / J_{\mu}, \, \af)$
is purely infinite by the same argument
as for $\ld = \ld_0.$
Therefore $C^* (G, \, J_{\mu + 1}, \, \alpha)$ is purely infinite
by Theorem~4.19 of~\cite{KR}.

If $\ld$ is a limit ordinal,
then
\[
C^* (G, J_{\ld}, \alpha)
  = {\overline{\bigcup_{\mu < \ld} C^* (G, J_{\mu}, \alpha)}},
\]
which is a direct limit of purely infinite C*-algebras
by the induction hypothesis,
and therefore is purely infinite by
Proposition~4.18 of~\cite{KR}.

This completes the induction.
Thus $C^* (G, A, \af)$ is purely infinite.
For~$A^{\af},$
apply Proposition~4.17 of~\cite{KR} and Proposition~\ref{P-FP303}.
\end{proof}

We give an example for Theorem~\ref{T-CompSer-304} in which
the C*-algebra is not a direct sum of C*-algebras
with finite primitive ideal spaces.

\begin{exa}\label{E-2801CmpSer}
Adopt the conventions for graph C*-algebras described in Chapter~1
of~\cite{Rb}.
(Warning: the papers \cite{BPRS}, \cite{HS}, and~\cite{KP},
all used below, use the opposite convention for the relation
between the direction
of the arrows in the graph and the definition of its C*-algebra.)
Consider the following graph~$E$:

\centerline{
\includegraphics[width=10cm]{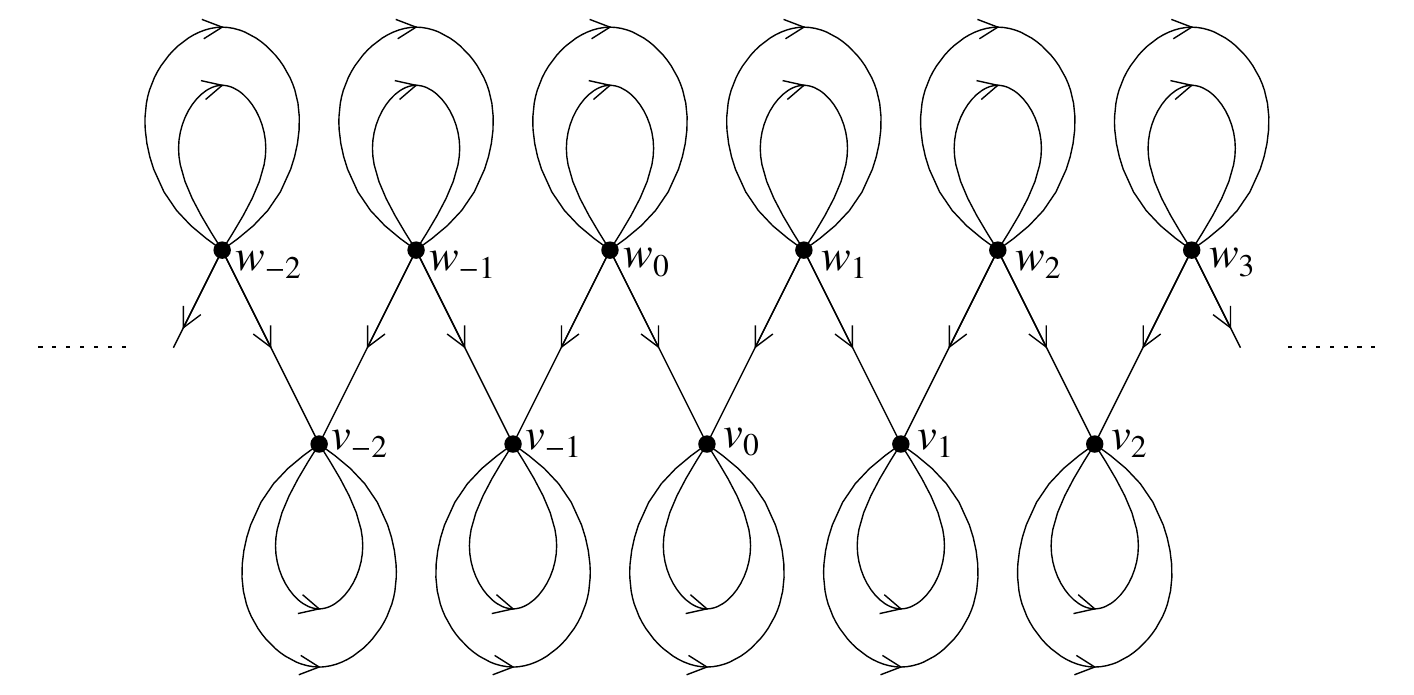}
}
\noindent
Let $C^* (E)$ be its C*-algebra.
The automorphism group of~$E$ acts on $C^* (E),$
as described,
for example,
in the introduction to Section~3 of~\cite{KP}.
There is a unique automorphism $h$ of $E$ of order~$2$
such that
$h (v_n) = v_{-n}$ and $h (w_n) = w_{1 - n}$ for all $n \in \Z,$
and which sends the outer loop at each vertex $x$
to the outer loop at $h (x).$
This gives an automorphism $\af$ of $C^* (E)$ of order~$2.$

We now check the properties of $C^* (E).$
We can check pure infiniteness using Theorem~2.3 of~\cite{HS},
but for this graph it seems easier to proceed more directly.
The graph~$E$
satisfies Condition~(K) as defined before Examples~4.6 of~\cite{Rb}
and
is row finite.
Therefore Theorem~4.9 of~\cite{Rb} applies.
The set $H = \{ w_n \colon n \in \Z \}$ is a saturated hereditary subset
of the vertices,
so this theorem gives, using the notation there, a short exact sequence
\[
0 \longrightarrow I_H
  \longrightarrow C^* (E)
  \longrightarrow C^* (E \setminus H)
  \longrightarrow 0,
\]
in which $I_H$ has a full corner isomorphic to $C^* (E_H).$
Both $E \setminus H$ and $E_H$
are isomorphic to countable disjoint unions
of copies of the graph with one vertex and two edges,
whose C*-algebra is~${\mathcal{O}}_2.$
Therefore $C^* (E \setminus H) \cong \oplus_{n \in \Z} {\mathcal{O}}_2$
and $I_H$ is stably isomorphic to $\oplus_{n \in \Z} {\mathcal{O}}_2.$
It now follows from Theorem~4.19 of~\cite{KR}
that $C^* (E)$ is purely infinite.

The graph $E$ is not a disjoint union of nonempty graphs,
from which it is easy to see that $C^* (E)$ does not have
a nontrivial direct sum decomposition.
The sets
\[
H_n = \big\{ v_{- n}, v_{- n + 1}, \ldots, v_{n - 1},
           w_{- n}, w_{- n + 1}, \ldots, w_{n} \big\},
% H_0 = \{ w_0 \},
% \,\,\,\,\,\,
% H_1 = \{ v_{-1}, v_0, w_{-1}, w_0, w_1 \},
% \,\,\,\,\,\,
% H_2 = \{ v_{-2}, v_{-1}, v_0, v_1, w_{-2}, w_{-1}, w_0, w_1, w_2 \},
% \,\,\,\,\,\,
% \ldots
\]
for $n \in \Nz,$
form an increasing sequence of finite saturated hereditary subsets
of the vertices of~$E$
whose union is all the vertices of~$E.$
The corresponding ideals each have finitely many ideals
(using Theorem~4.19 of~\cite{KR}).
Their union is dense in~$C^* (E),$
since, by Theorem~4.19 of~\cite{KR},
the closure of the union comes from some saturated hereditary subset
of the vertices of~$E.$
Therefore $C^* (E)$ has a composition series
in which all the subquotients have finite primitive ideal spaces.
\end{exa}

We didn't explicitly need the space $\Prim (C^* (E)),$
but it can easily be calculated using Theorem~6.3 of~\cite{BPRS}.

There are many other automorphisms of the graph $E$ in
Example~\ref{E-2801CmpSer} of order two.
For example, one can modify $h$ by having it exchange the two loops
at $v_0$ instead of acting trivially on them.
One similarly gets two automorphisms of order two which act on
the vertices via $v_n \mapsto v_{- n - 1}$ and $w_n \mapsto w_{- n}.$
One gets automorphisms of order four by specifying that for some
values of $n > 0,$
the automorphism should send the outer loop at $v_n$ or at $w_n$
to the inner loop at $h (v_n)$ or $h (w_n).$

Free actions on graphs are the subject of a very nice theorem
in~\cite{KP},
and we can make a similar example with a free action of~$\Z_2.$
Let $F$ be the following graph:

\centerline{
\includegraphics[width=12cm]{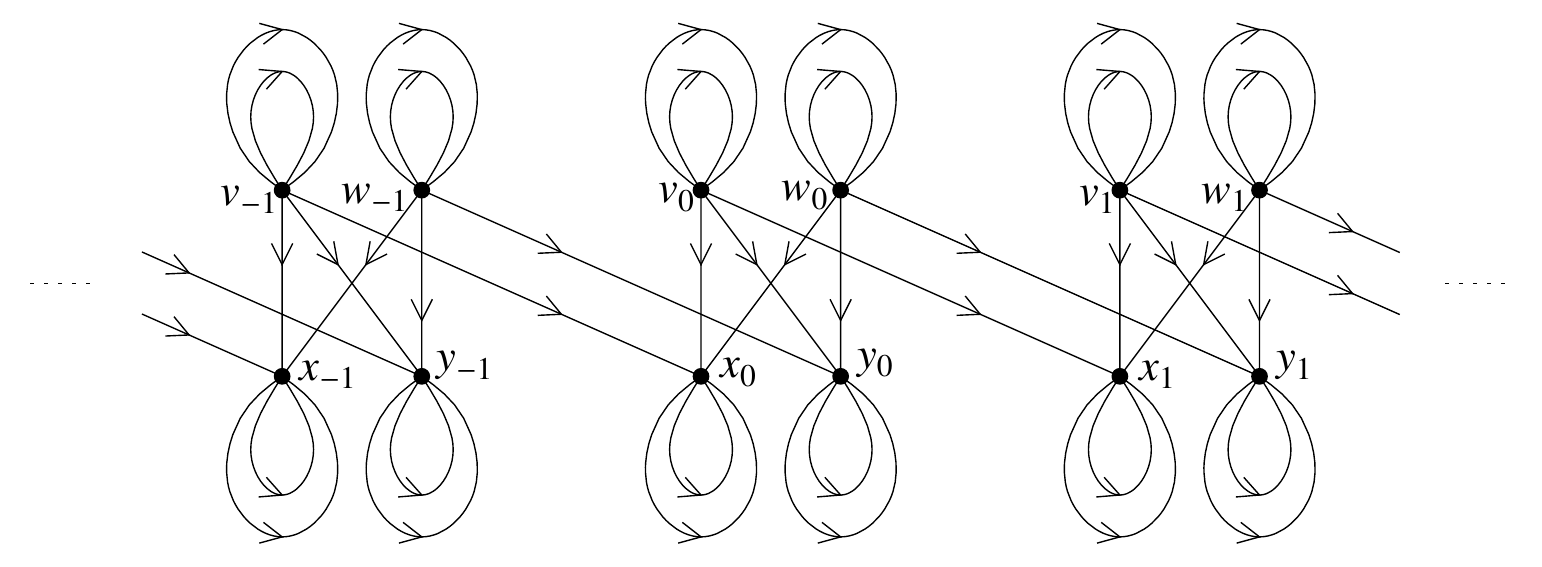}
}
\noindent
Then $C^* (F)$ is purely infinite and has a composition series
whose subquotients have finite primitive ideal spaces
by the same reasoning as in Example~\ref{E-2801CmpSer}.
There is an automorphism $h \colon F \to F$ of order~$2$
which acts on the vertices by
\[
h (v_n) = w_n,
\,\,\,\,\,\,
h (w_n) = v_n,
\,\,\,\,\,\,
h (x_n) = y_n,
\andeqn
h (y_n) = x_n
\]
for $n \in \Z,$ and the corresponding action of $\Z_2$ on~$F$ is free.

A proof related to that of Theorem~\ref{T-CompSer-304}
also gives the following result.
The hypotheses are stronger than those of Theorem~\ref{T-CompSer-304},
but still apply to a C*-algebra
which is an arbitrary (not necessarily finite) direct sum
of purely infinite C*-algebras with finite primitive ideal spaces,
and to a purely infinite C*-algebra with a composition series
indexed by $\Nz$ and with simple subquotients.
Example~\ref{E-2801CmpSer} and the related examples
discussed afterwards also satisfy its hypotheses.
The conclusion is stronger,
since it includes the ideal property for both the crossed product
and the fixed point algebra.
In particular, pure infiniteness
rules out the phenomenon in Example~\ref{E:TwoStep}.

\begin{prp}\label{P:P1_0126}
Let $\af \colon G \to \Aut (A)$ be an action of a finite group $G$
on a C*-algebra~$A.$
Suppose that there is a set ${\mathcal{I}}$ of ideals in~$A,$
each of which is purely infinite and has finite primitive ideal space,
with the following property.
For every finite subset $S \subset A$ and every $\ep > 0,$
there is $I \in {\mathcal{I}}$
such that $\dist (a, \, I) < \ep$ for all $a \in S.$
Then $C^* (G, A, \af)$ and $A^{\af}$
are purely infinite and have the ideal property.
\end{prp}

\begin{proof}
We use the notation $G I$ from the proof of Theorem~\ref{T-CompSer-304}.
Let $\Ld$ be the set of all finite subsets of~${\mathcal{I}}.$
For $\ld \in \Ld,$
set $J_{\ld} = \sum_{I \in \ld} G I.$
Then $\Prim (J_{\ld})$ is finite by Corollary~\ref{C-OldL2_0123}.
An argument similar to one used in the proof of
Theorem~\ref{T-CompSer-304}
shows that $C^* (G, J_{\ld}, \alpha)$ is purely infinite,
so Proposition~4.17 of~\cite{KR} and Proposition~\ref{P-FP303}
imply that $(J_{\ld})^{\af}$ is purely infinite.
Theorem~\ref{L:L3_0114} also implies that
$\Prim ( C^* (G, J_{\ld}, \alpha) )$ is finite,
so Proposition~\ref{P-FP303} implies that
$\Prim ( (J_{\ld})^{\af} )$ is finite.
So $C^* (G, J_{\ld}, \alpha)$ and $(J_{\ld})^{\af}$
have the ideal property by Proposition~2.3 of~\cite{KNP}.

We have $A = \dirlim_{\ld \in \Ld} J_{\ld},$
so $C^* (G, A, \af) = \dirlim_{\ld \in \Ld} C^* (G, J_{\ld}, \alpha),$
which is then purely infinite by Proposition~4.18 of~\cite{KR}.
Also, $C^* (G, A, \af)$ has the ideal property
by Proposition~2.3 of~\cite{Psn1}.
Since $A^{\af} = \dirlim_{\ld \in \Ld} (J_{\ld})^{\af},$
the same reasoning shows that
$A^{\af}$ is purely infinite and has the ideal property.
\end{proof}

\end{document}